%% file: RKLapprox_reviesed_v2.tex
\documentclass[final,twoside]{siamart1116}

\usepackage{graphics} % for pdf, bitmapped graphics files
\usepackage{amsmath} % assumes amsmath package installed
\usepackage{amssymb}  % assumes amsmath package installed

\usepackage{listings}
\usepackage{enumerate}
\usepackage{booktabs}

\usepackage{tikz}
\usetikzlibrary{arrows}
\usepackage{algorithm}
\usepackage{algorithmic}

\usepackage{flushend}

\usepackage{url}

\usepackage[noadjust]{cite}

\usepackage[caption=false]{subfig}

\numberwithin{theorem}{section}
\numberwithin{equation}{section}

\newenvironment{example}{\refstepcounter{theorem} {\em Example} \thetheorem.}{}
\newenvironment{remark}{\refstepcounter{theorem} {\em Remark} \thetheorem.}{}

\newcommand{\derivd}{d}
\newcommand{\dmuh}{\derivd \nu}
\newcommand{\dmu}{\derivd \mu}
\newcommand{\dm}{\derivd m}
\newcommand{\dmuhat}{\derivd \hat{\mu}}

\DeclareMathOperator*{\argmin}{arg\,min}

\DeclareMathOperator*{\supp}{supp}

\DeclareMathOperator*{\st}{subject\;to}

\newcommand{\mZ}{\mathbb{Z}}

\newcommand{\mD}{\mathbb{D}}
\newcommand{\mJ}{\mathbb{J}}
\newcommand{\mI}{\mathbb{I}}
\newcommand{\mR}{\mathbb{R}}
\newcommand{\mT}{\mathbb{T}}

\newcommand{\fC}{\mathfrak{C}}

\newcommand{\PP}{\mathfrak{P}_+}
\newcommand{\PPC}{\bar{\mathfrak{P}}_+}

\newcommand{\CP}{\mathfrak{C}_+}
\newcommand{\CPC}{\bar{\mathfrak{C}}_+}
\newcommand{\SW}{\mathfrak{S}_W}
\newcommand{\ScW}{\mathfrak{S}_{c,W}}

\newcommand{\thetab}{{\boldsymbol \theta}}
\newcommand{\kb}{{\boldsymbol k}}
\newcommand{\nollb}{{\boldsymbol 0}}

\newcommand{\expfunt}{e^{i(\tb,\thetab)}}
\newcommand{\expfunk}{e^{i(\kb,\thetab)}}
\newcommand{\expfunkm}{e^{-i(\kb,\thetab)}}
\newcommand{\tb}{{\boldsymbol t}}

\newcommand{\kl}{\mathbb D}

\DeclareMathOperator*{\ExpOp}{\mathbb{E}}

\makeatletter
\def\iddots{\mathinner{\mkern1mu\raise\p@
\vbox{\kern7\p@\hbox{.}}\mkern2mu
\raise4\p@\hbox{.}\mkern2mu\raise7\p@\hbox{.}\mkern1mu}}
\makeatother

\def \myFigPath {figures/}

\def \spect_est_fig {spect_est/}

\def \textureOne {"text_gen/granite001-inca-100dpi-00_reg_0point01/"}
\def \textureTwo {"text_gen/paper010-inca-100dpi-00_reg_0point01/"}
\def \textureThree {"text_gen/plastic008-inca-100dpi-00_reg_0point01/"}

\title{
Multidimensional Rational Covariance Extension\\ with Approximate Covariance Matching%
\thanks{This work was supported by the Swedish Research Council (VR), the Swedish Foundation of Strategic Research (SSF), and the ACCESS Linnaeus Center, KTH  Royal Institute of Technology.}
} 

\author{Axel Ringh\footnotemark[2] \and Johan Karlsson\footnotemark[2] \and Anders Lindquist\footnotemark[3] \footnotemark[2]}

\newcommand{\TheTitle}{Approximate Multidimensional Rational Covariance Extension} 
\newcommand{\TheAuthors}{A.~Ringh, J.~Karlsson, and A.~Lindquist}

\headers{\TheTitle}{\TheAuthors}

\begin{document}

% The affiliations and fundings
\maketitle
\slugger{sicon}{xxxx}{xx}{x}{x--x}%slugger should be set to mms, siap, sicomp, sicon, sidma, sima, simax, sinum, siopt, sisc, or sirev

\renewcommand{\thefootnote}{\fnsymbol{footnote}}
\footnotetext[2]{Division of Optimization and Systems Theory, Department of Mathematics, KTH  Royal Institute of Technology, 100 44 Stockholm, Sweden. (\email{aringh@kth.se}, \email{johan.karlsson@math.kth.se})}
\footnotetext[3]{Department of Automation and School of Mathematics, Shanghai Jiao Tong University,
200240 Shanghai, China. (\email{alq@kth.se})}

% To be able to use normal footnotes later on in the file
\renewcommand*{\thefootnote}{\arabic{footnote}}
\setcounter{footnote}{0}

% From here on the actual document starts
\begin{abstract}
In our companion paper \cite{ringh2015multidimensional} we discussed the multidimensional rational covariance extension problem (RCEP), which has important applications in image processing, and spectral estimation in radar, sonar, and medical imaging. This is an inverse problem where a power spectrum with a rational absolutely continuous part  is reconstructed from a finite set of moments. However, in most applications these moments are determined from observed data and are therefore only approximate, and RCEP may not have a solution. In this paper we extend the results \cite{ringh2015multidimensional} to handle approximate covariance matching. We consider two problems, one with a soft constraint and the other one with a hard constraint, and show that they are connected via a homeomorphism. We also demonstrate that the problems are well-posed and illustrate the theory by examples in spectral estimation and texture generation.
%
%\bigskip
%
%---------------
%
%\bigskip
%
% \ar{Add: Something about illustrated on realization/identification example for textures/materials XX Remove if example is removed: This approach is compared with exact matching to the biased covariance estimate, which is extendable.}
% 
\end{abstract}

\begin{keywords}Approximate covariance extension, trigonometric moment problem, convex optimization, multidimensional spectral estimation, texture generation.\end{keywords}

%\begin{AMS}\end{AMS}

\section{Introduction}\label{sec:intro}

Trigonometric moment problems are ubiquitous in systems and control, such as spectral estimation, signal processing, system identification, image processing and remote sensing \cite{bose2003multidimensional,ekstrom1984digital,stoica1997introduction}. In the  (truncated) multidimensional trigonometric moment problem we seek a nonnegative measure $\dmu$ on $\mT^d$ satisfying  the moment equation 
\begin{equation} \label{eq:Cov}
c_\kb = \int_{\mT^d} \expfunk \dmu(\thetab) \quad\text{for all $\kb\in\Lambda$},
\end{equation}
where $\mT:=(-\pi,\pi]$, $\thetab:=(\theta_1,\ldots, \theta_d)\in\mT^d$, and $(\kb,\thetab):=\sum_{j=1}^d k_j\theta_j$ is the scalar product in $\mR^d$. Here $\Lambda \subset \mathbb{Z}^d$ is a finite index set satisfying $0 \in \Lambda$ and $- \Lambda = \Lambda$. A necessary condition for \eqref{eq:Cov} to have a solution is that the sequence 
\begin{equation}
\label{eq:covariances}
c:= [ c_\kb \mid \kb:=(k_1,\ldots, k_d) \in \Lambda ]
\end{equation}
satisfy the symmetry condition $c_{-\kb}=\bar{c}_\kb$. The space of sequences \eqref{eq:covariances} with this symmerty  will be denoted $\fC$ and will be represented by vectors  $c$, formed by ordering the coefficient in some prescribed manner, e.g., lexiographical. Note that $\fC$ is isomorphic to $\mR^{|\Lambda|}$, where $|\Lambda|$ is the cardinality of $\Lambda$. 
However, as we shall see below, not all $c\in\fC$ are {\em bona fide\/} moments for nonnegative measures $\dmu$. 

In many  of the applications mentioned above there is a natural complexity constraint prescribed by design specifications. In the context of finite-dimensional systems these constraints often arise in the requirement that transfer functions be rational. This leads to the {\em rational covariance extension problem}, which  has been  studied in various degrees of generality in \cite{georgiou2005solution,georgiou2006relative,karlsson2015themultidimensional,ringh2015themultidimensional,ringh2015multidimensional} and can be posed as follows. 

Define $e^{i\thetab}:=(e^{i\theta_1},\ldots, e^{i\theta_d})$ and let
\begin{subequations}\label{eq:ratinaldmu}
\begin{equation}\label{eq:decomp}
\dmu(\thetab) = \Phi(e^{i\thetab})\dm(\thetab) + \dmuh(\thetab),
\end{equation}
be the (unique) Lebesgue decomposition of $d\mu$  (see, e.g., \cite[p. 121]{rudin1987real}), 
where $$dm (\thetab):=(1/2\pi)^d\prod_{j=1}^d d\theta_j$$ is the (normalized)
Lebesgue measure and $\dmuh$ is a singular measure. Then given a $c\in\fC$, we are interested in parameterizing solutions to \eqref{eq:Cov} such that the absolutely continuous part of the measure \eqref{eq:decomp}  takes the form
\begin{equation}\label{eq:rat}
\Phi(e^{i\thetab})=\frac{P(e^{i\thetab})}{Q(e^{i\thetab})},\quad p, q \in \PPC \backslash \{0\},
\end{equation}
\end{subequations}
where $\PPC$ is the closure of the convex cone $\PP$ of the coefficients $p\in\fC$ corresponding to trigonometric polynomials
\begin{equation}\label{Ptrigpol}
%\PP=\{p=[p_{\kb}]_{\kb\in \Lambda} \mid 
P(e^{i\thetab})=
\sum_{\kb\in \Lambda}p_\kb e^{-i(\kb,\thetab)}, \quad p_{-\kb}=\bar{p}_\kb
\end{equation}
that are positive for all $\thetab\in \mT^d$. 

The reason for referring to this problem as a rational covariance extension problem is that the numbers \eqref{eq:covariances} correspond to covariances $c_\kb:= \ExpOp \{y({\bf t}+\kb)\overline{y({\bf t})}\}$ of a  discrete-time, zero-mean, and homogeneous\footnote{Homogeneity  generalizes stationarity in the case $d=1$. }  stochastic process  $\{y({\bf t});\,{\bf t}\in\mZ^d\}$. The corresponding power spectrum, representing the energy distribution across frequencies, is defined as the nonnegative measure $d\mu$ on $\mT^d$ whose Fourier coefficients are the covariances \eqref{eq:covariances}.
A scalar version of this problem ($d=1$) was first posed by Kalman \cite{kalman1981realization} and has been extensively studied and solved in the literature \cite{georgiou1983partial, byrnes1995acomplete, byrnes2002identifyability, enqvist2004aconvex, nurdin2006new, CarliGeorgiou2011, lindquist2013thecirculant, byrnes2000anewapproach, zorzi2014rational, musicus1985maximum}.
It has been generalized to more general scalar moment problems \cite{byrnes2001ageneralized, georgio2003kullback, byrnes2006generalizedinterpolation} and to the multidimensional setting \cite{georgiou2006relative, georgiou2005solution, ringh2015multidimensional, ringh2015themultidimensional, karlsson2015themultidimensional}. Also worth mentioning here is work by Lang and McClellan \cite{lang1982multidimensional, lang1983spectral, mcclellan1982multi-dimensional, mcclellan1983duality, lang1982theextension, lang1981spectral} considering the multidimensional maximum entropy problem, which hence has certain overlap with the above literature.

The multidimensional rational covariance extension problem posed above has a solution if and only if $c\in\CP$, where $\CP$ is the open convex cone 
\begin{equation*}
\CP := \left\{ c \mid \langle c, p \rangle > 0, \quad \text{for all $p \in \PPC \setminus \{0\}$} \right\},
\end{equation*}
where $\langle c, p \rangle := \sum_{\kb \in \Lambda} c_\kb \bar{p}_\kb$ is the inner product in $\fC$ (Theorem~\ref{thm:rational}).  However, the covariances $[ c_\kb \mid \kb:=(k_1,\ldots, k_d) \in \Lambda]$ are generally determined from statistical  data. Therefore the condition $c\in\CP$ may not be satisfied, and testing this condition is difficult in the multidimensional case. 
Therefore we may want to find a positive measure $\dmu$ and  a corresponding $r\in\CP$, namely
\begin{equation}
\label{eq:r}
r_\kb = \int_{\mT^d} \expfunk \dmu(\thetab), \quad \kb\in\Lambda,
\end{equation}
so that $r$ is close  to $c$ in some norm, e.g.,  the Euclidian norm $\|\centerdot\|_2$. This is an ill-posed inverse problem which in general has an infinite number of  solutions $\dmu$. As we already mentioned, we are interested in rational solutions  \eqref{eq:ratinaldmu}, and to obtain such solutions we use  regularization as in \cite{ringh2015multidimensional}. Hence, we seek a $d\mu$  that minimizes
\begin{displaymath}
\lambda  \kl (P\dm, \dmu)+\frac12\|r-c\|_2^2
\end{displaymath}
subject to \eqref{eq:r},
where $\lambda\in\mR$ is a regularization parameter and
\begin{equation}\label{eq:normalizedKullbackLeibler}
 \kl (P\dm, \dmu) :=  \int_{\mT^d}  \left(  P \log\frac{P}{\Phi} \dm +d\mu -Pdm \right)
\end{equation}
is the nomalized Kullback-Leibler divergence \cite[ch. 4]{gzyl1995method} \cite{csiszar1991why, zorzi2014rational}. As will be explained in Section~\ref{sec:prel}, $ \kl (P\dm, \dmu)$ is always nonnegative and has the property $\kl (P\dm,P\dm)=0$. 

In this paper we shall consider a more general problem in the spirit of \cite{enqvist2007approximative}. To this end, for any Hermitian, positive definite matrix $M$,  we define the weighted vector norm $\|x\|_M:=(x^*Mx)^{1/2}$ and consider the problem
\begin{align}\label{eq:primal_relax}
\min_{\dmu \geq 0, \, r}  \qquad & \kl (P\dm, \dmu) + \frac{1}{2}\|r - c\|_{W^{-1}}^2 \\
\st  \quad &r_\kb = \int_{\mT^d} \expfunk \dmu(\thetab), \quad \kb\in\Lambda, \nonumber
\end{align}
which is the same as the problem above with $W=\lambda I$. We shall refer to $W$ as the {\em weight matrix}.

Using the same principle as in \cite{schott1984maximum}, we shall also consider the problem to minimize $\kl (P\dm, \dmu)$ subject to \eqref{eq:r} and the hard constraint 
\begin{equation}
\label{eq:hardconstraint}
\|r - c \|^2 \le \lambda \,.
\end{equation}
Since \eqref{eq:r} are {\em bona fide\/} moments and hence $r\in\CP$, while $c\not\in\CP$ in general, this problem will not have a solution if the distance from $c$ to $\CP$ is greater than $\sqrt{\lambda}$. Hence the choice of $\lambda$ must be made with some care. Analogously with the {\em rational covariance extension with soft  constraints\/} in \eqref{eq:primal_relax}, we shall consider the more general problem 
\begin{align}\label{eq:primal_relax_new}
\min_{\dmu \geq 0, \, r}  \qquad &\kl (P\dm, \dmu)  \\
\st  \quad &r_\kb = \int_{\mT^d} \expfunk \dmu(\thetab), \quad \kb\in\Lambda, \nonumber \\
     \quad &\|r - c \|_{W^{-1}}^2 \le 1, \nonumber
\end{align}
to which we shall refer as the {\em rational covariance extension problem with hard constraints}. Again this problem reduces to the simpler problem by setting $W=\lambda I$.

As we shall see, the soft-constrained problem \eqref{eq:primal_relax} always has a solution, while the hard-constrained problem  \eqref{eq:primal_relax_new} may fail to have a solution for some weight matrices $W$. However, in Section~\ref{sec:homeomorphism} we show that the two problems are in fact equivalent in the sense that whenever \eqref{eq:primal_relax_new} has a solution there is a corresponding $W$ in \eqref{eq:primal_relax} that gives the same solution, and any solution of \eqref{eq:primal_relax} can also be obtained from \eqref{eq:primal_relax_new} by a suitable choice of $W$. The reason for considering both formulations is that  one formulation might be more suitable than the other for the particular application at hand. For example, an absolute error estimate for the covariances is more naturally incorporated in the formulation with hard constraints. A possible choice of the weight matrix $W$ in either formulation would be the covariance matrix of the estimated moments, as suggested in \cite{enqvist2007approximative}. This corresponds to the Mahalanobis distance  and could be a natural way to incorporate uncertainty of the covariance estimates in the spectral estimation procedure.

Previous work in this direction can be found in \cite{shankwitz1990onthe, enqvist2007approximative, byrnes2003theuncertain, schott1984maximum, karlsson2016confidence}, 
%where \cite{schott1984maximum} and \cite{enqvist2007approximative} consider two different kinds of relaxations, 
where \cite{shankwitz1990onthe,karlsson2016confidence,byrnes2003theuncertain} consider the problem of selecting an appropriate covariances sequence to match in a given confidence region. The two approximation problems considered here are similar to the ones considered in \cite{schott1984maximum} and \cite{enqvist2007approximative}. (For more details, also see \cite[Ch. B]{avventi2011spectral}.)

We begin in Section~\ref{sec:prel}  by reviewing the regular multidimensional rational covariance extension problem for exact covariance matching  in a broader perspective. In Section~\ref{sec:approxsoft} we present our main results on approximate rational covariance extension with soft  constraints, and in Section~\ref{sec:wellposed} we show that the dual solution is well-posed. In Section~\ref{sec:no_singular_part} we investigate conditions under which there are solutions without a singular part. The approximate rational covariance extension with hard  constraints is considered in 
Section~\ref{sec:approxhard}, and in Section~\ref{sec:homeomorphism} we establish a homeomorhism between the weight matrices in the two problems, showing that the problems are actually equivalent when solutions exist. We also show that under certain conditions the homeomorphism can be extended to hold between all sets of parameters, allowing us to carry over results from the soft-constrained setting to the hard-constrained one.  In Section~\ref{sec:covest} we discuss the properties of various covariance estimators, in Section~\ref{sec:example} we give a 2D example from spectral estimation, and in Section~\ref{sec:ex_texture}  we apply our theory to system identification and texture reconstruction. 
Some of the results of this paper were announced in \cite{ringh2016multidimensional} without proofs.

%Again, we denote the boundary by $\partial \CP$. In the one-dimensional case ($d=1$) the condition $c\in\CPC$ amounts to the corresponding Toeplitz matrix being positive definite \cite{krein1977themarkov}, but in the multidimensional case ($d>1$) no such simple matrix condition is known. 

\section{Rational covariance extension with exact matching}\label{sec:prel}

The trigonometric moment problem to determine a positive measure $\dmu$ satisfying \eqref{eq:Cov} is an inverse problem that has a solution if and only if $c\in\CPC$ \cite[Theorem 2.3]{karlsson2015themultidimensional}, where $\CPC$ is the closure of $\CP$, and then in general it has infinitely many solutions. However, the nature of possible rational solutions \eqref{eq:ratinaldmu} will depend on the location of $c$ in $\CPC$. To clarify this point we need the following lemma.

\begin{lemma}\label{lem:PsubsetC}
$\PPC \setminus \{0\} \subset \CP$.
\end{lemma}

\begin{proof}
Obviously the inner product $\langle q, p \rangle := \sum_{\kb \in \Lambda} q_\kb \bar{p}_\kb$ can be expressed in the integral form
\begin{equation}
\label{innerprod}
\langle q, p \rangle = \int_{\mT^d} Q(e^{i\thetab})\overline{P(e^{i\thetab})}\dm(\thetab),
\end{equation}
and therefore $\langle q,p\rangle >0$ for all $q,p\in\PPC \setminus \{0\}$, as $P$ and $Q$ can have zeros only on sets of measure zero. Hence the statement of the lemma follows.
\end{proof}

Therefore, under certain particular conditions, the multidimensional rational covariance extension problem has a very simple solution with a polynomial spectral density, namely
\begin{equation}
\label{eq:dmu=Pdm}
\dmu= P(e^{i\thetab})dm(\thetab), \quad p \in \PPC \backslash \{0\}. 
\end{equation}

\begin{proposition}\label{prop:polsolution}
The multidimensional rational covariance extension problem has a unique polynomial solution \eqref{eq:dmu=Pdm} if and only if $c\in\PPC \backslash \{0\}$, namely $P=C$, where 
\begin{displaymath}
C(e^{i\thetab}):=\sum_{\kb\in \Lambda}c_\kb e^{-i(\kb,\thetab)}.
\end{displaymath}
\end{proposition}

The proof of Proposition~\ref{prop:polsolution} is immediate by noting that any such $C$ is a {\em bona fide\/} spectral density and noting that  $c_\kb=\int_{\mT^d}e^{i(\kb,\thetab)}C(e^{i\thetab})dm(\thetab)$. 

As seen from the following result presented in \cite[Section 6]{karlsson2015themultidimensional}, the other extreme occurs for $c\in\partial\CP:=\CPC\setminus\CP$, when only singular solutions exist.

\begin{proposition}\label{prop:singular}
For any $c\in\partial\CP$ there is a solution $\dmu$ of \eqref{eq:Cov} with support in at most $|\Lambda|-1$ points. There is no solution with a absolutely continuous part $\Phi\dm$. 
\end{proposition}

However, for any $c\in\CP$, there is a rational solution \eqref{eq:ratinaldmu} parametrized by $p\in\PPC\backslash\{0\}$, as demonstrated in \cite{ringh2015multidimensional} by considering a primal-dual pair of convex optimization problems.  In that paper the primal problem is a weighted maximum entropy problem, but as also noted in \cite[Sec. 3.2]{ringh2015multidimensional}, it is equivalent to
%\begin{align}\label{eq:primal}
%\max_{\dmu \geq 0} & \quad \int_{\mT^d} P \log\Phi \dm(\thetab) \\
%\st & \quad c_\kb = \int_{\mT^d} \expfunk \dmu(\thetab), \quad \kb \in \Lambda, \nonumber
%\end{align}
\begin{align}\label{eq:primal}
\min_{\dmu \geq 0} & \quad \int_{\mT^d} P \log \frac{P}{\Phi} \dm(\thetab) \\
\st & \quad c_\kb = \int_{\mT^d} \expfunk \dmu(\thetab), \quad \kb \in \Lambda, \nonumber
\end{align}
where $\Phi dm$ is the absolutely continuous part of $d\mu$.
This amounts to minimizing the (regular) Kullback-Leibler divergence between $P\dm$ and $\dmu$, subject to $\dmu$ matching the given data \cite{georgio2003kullback,ringh2015multidimensional}. In the present case of exact covariance matching, this problem is equivalent to minimizing \eqref{eq:normalizedKullbackLeibler} subject to \eqref{eq:Cov}, since $P$ is fixed and the total mass of $\dmu$ is determined by the $0$:th moment  $c_{\boldsymbol 0} = \int_{\mT^d} \dmu$. 
Hence both $\int_{\mT^d} \dmu$ and $\int_{\mT^d} P\dm$ are constants in this case. Hence problem \eqref{eq:primal_relax} is the natural extension of \eqref{eq:primal} for the case where the covariance sequence is not known exactly. 

The primal problem \eqref{eq:primal} is a problem in infinite dimensions, but with a finite number of constraints. The dual to this problem will then have a finite number of variables but an infinite number of constraints and is given by
\begin{align}\label{eq:dual}
\min_{q\in\PPC} & \quad \langle c, q \rangle - \int_{\mT^d} P \log Q \,\dm(\thetab). 
%\\
%&\st  \quad Q \in \PPC. \nonumber
\end{align}
In particular, Theorem  2.1 in \cite{ringh2015multidimensional}, based on corresponding analysis in \cite{karlsson2015themultidimensional}, reads as follows.

\begin{theorem}\label{thm:rational}
Problem \eqref{eq:primal} has a solution if and only if $c \in \CP$. For every $c \in \CP$ and $p \in \PPC \setminus \{0\}$ the functional in \eqref{eq:dual} is strictly convex and has a unique minimizer $\hat{q} \in \PPC \setminus \{0\}$. Moreover, there exists a unique $\hat{c} \in \partial \CP$ and a (not necessarily unique) nonnegative singular measure $\derivd\hat\nu$ with support 
\begin{equation}
\label{supp(dnu)}
\supp(\derivd\hat\nu) \subseteq \{ \thetab \in \mT^d  \mid \hat{Q}(e^{i\thetab}) = 0 \}
\end{equation}
such that
\begin{subequations}
\begin{align}
& c_{\kb} = \int_{\mT^d} \expfunk \left( \frac{P}{\hat Q}  \dm  + \derivd\hat\nu \right),\quad \kb \in \Lambda, \\
& \hat{c}_\kb = \int_{\mT^d} \expfunk \derivd\hat\nu, \quad\kb \in \Lambda.
\end{align}
\end{subequations}
For any such $\derivd\hat\nu$, the measure 
\begin{equation}
\label{ }
\derivd\hat\mu(\thetab) = \frac{P(e^{i\thetab})}{\hat{Q}(e^{i\thetab})}\dm(\thetab) + \derivd\hat\nu(\thetab)
\end{equation}
is an optimal solution to the problem  \eqref{eq:primal}. Moreover, $\derivd\hat\nu$ can be chosen with support in at most $|\Lambda|-1$ points, where $|\Lambda|$ is the cardinality of the index set $\Lambda$.
\end{theorem}

If $c\in \partial \CP$, only a singular measure with finite support would match the moment condition (Proposition~\ref{prop:singular}). In this case, the problem \eqref{eq:primal} makes no sense, since any feasible solution has infinite objective value.  

In \cite{karlsson2015themultidimensional} we also derived the KKT conditions
\begin{subequations}\label{eq:opt_cond_exact}
\begin{align}
& \hat{q} \in \PPC, \quad \hat c \in \partial \CP, \quad \langle \hat{c}, \hat{q} \rangle = 0 \label{eq:opt_cond_exact_slack} \\
& c_\kb = \int_{\mT^d} \expfunk \frac{P}{\hat Q}\dm + \hat{c}_\kb,\quad \kb \in \Lambda  ,\label{eq:opt_cond_exact_KKT}
\end{align}
\end{subequations}
which are necessary and sufficient for optimality of the primal and dual problems. 

Since  \eqref{eq:primal}  is an inverse problem, we are interested in how the solution depends on the parameters of the problem. From Propositions 7.3 and 7.4 in \cite{karlsson2015themultidimensional} we have the following result.

\begin{proposition}\label{prop:cp2qhat}
Let $c$,  $p$ and $\hat{q}$  be as in Theorem~\ref{thm:rational}. Then the map $(c,p)\mapsto \hat{q}$ is continuous.
\end{proposition}

To get a full description of well-posedness of the solution we would like to extend this continuity result to the map $(c,p)\mapsto (\hat{q},\hat{c})$. However, such a generalization is only possible under certain conditions. The following result is a consequence of Proposition~\ref{prop:cp2qhat} and \cite[Corollary 2.3]{ringh2015multidimensional}.

\begin{proposition}\label{prop:c2qhatchat2}
Let $c$, $p$, $\hat{q}$ and $\hat{c}$ be as in Theorem~\ref{thm:rational}. Then,  for  $d\leq 2$ and all $(c,p)\in\CP\times\PP$, the mapping $(c,p)\to(\hat q, \hat c)$  is continuous.  
\end{proposition}

Corollary 2.3 in \cite{ringh2015multidimensional} actually ensures that $\hat{c}=0$ for $d\leq 2$ and $p\in\PP$. However, in Section~\ref{ssec:qhat2chat} we present a generalization of Proposition~\ref{prop:c2qhatchat2} to cases with $d\geq3$, where then $\hat{c}$ may be nonzero. (The proof of this generalization can be found in \cite{ringh2017further}.)
Here we shall also consider an example where continuity fails when $p$ belongs to the boundary $\partial\PP:= \PPC \setminus \PP$, i.e.,  the corresponding nonnegative trigonometric polynomial $P(e^{i\thetab})$ is zero in at least one point. 

%\section{Main results}\label{sec:mainresults}
\section{Approximate covariance extension with soft constraints}\label{sec:approxsoft}

To handle the case with noisy covariance data, when $c$ may not even belong to $\CP$, we  relax the exact covariance matching constraint \eqref{eq:Cov} in the primal problem \eqref{eq:primal} to obtain the problem \eqref{eq:primal_relax}. In this case  it is natural to reformulate the objective function in \eqref{eq:primal} to include a term that also accounts for changes in the total mass of $\dmu$. Consequently, we have exchanged the objective function in \eqref{eq:primal} by the normalized Kullback-Leibler divergence \eqref{eq:normalizedKullbackLeibler} plus a term that ensures approximate data matching.

Using the normalized Kullback-Leibler divergence, as proposed in  \cite[ch. 4]{gzyl1995method} \cite{csiszar1991why, zorzi2014rational}, is an advantage in the approximate covariance matching problem since this divergence is always nonnegative, precisely as is the case for probability densities. To see this, observe that, in view of the basic inequality  $x - 1 \geq \log x$,
\begin{align*}
  \kl(P\dm, \dmu) &= \int_{\mT^d}  \left(  P \left( - \log\frac{\Phi}{P} \right) \dm +d\mu -Pdm \right) \\
 & \geq \int_{\mT^d}  \left(  P (1 - \frac{\Phi}{P})\dm + \Phi\dm -Pdm \right) + \int_{\mT^d} \dmuh \geq 0,
 \end{align*}
 since $\dmuh$ is a nonnegative measure. 
Moreover, $ \kl(P\dm, P\dm)=0$, as can be seen by taking $\dmu=Pdm$ in \eqref{eq:normalizedKullbackLeibler}.

The problem under consideration is to find a nonnegative measure $\dmu=\Phi dm +\derivd\nu$ minimizing $$\kl(P\dm, \dmu)+\frac{1}{2}\|r - c\|_{W^{-1}}^2$$ subject to \eqref{eq:r}. To derive the dual of this problem we consider the corresponding maximization problem and form the Lagrangian 
\begin{align*}
\mathcal{L}(\Phi,\derivd\nu, r, q) = & \; -\kl(P\dm, \dmu) - \frac{1}{2}\|r - c\|_{W^{-1}}^2 + \sum_{\kb \in \Lambda} q_\kb^* \Big( r_\kb - \int_{\mT^d} \expfunk \dmu(\thetab) \Big)\\
=&\;-\kl(P\dm, \dmu) - \frac{1}{2}\|r - c\|_{W^{-1}}^2  +\langle r,q\rangle - \int_{\mT^d}Q\dmu \, ,
\end{align*}
where $q := [ q_\kb \mid \kb:=(k_1,\ldots, k_d) \in \Lambda]$ are Lagrange multitipliers and $Q$ is the corresponding trigonometric polynomial \eqref{Ptrigpol}. However, 
\begin{equation}
\label{D(Pdm,dmu)}
\kl(P\dm, \dmu)=\int_{\mT^d}P(\log P -1)dm - \int_{\mT^d}P\log\Phi dm + r_\nollb\, ,
\end{equation}
and therefore 
\begin{align}\label{eq:Lagrangian}
\mathcal{L}(\Phi,\derivd\nu, r, q) = &\; \int_{\mT^d}P\log\Phi dm -  \int_{\mT^d}Q\Phi dm -  \int_{\mT^d} Q\derivd\nu  -\int_{\mT^d}P(\log P -1)dm\notag \\
&\; +\langle r,q-e\rangle - \frac{1}{2}\|r - c\|_{W^{-1}}^2 \, , 
\end{align}
where  $e := [e_\kb]_{\kb \in \Lambda}$, $e_{\boldsymbol 0} = 1$ and $e_{\kb} = 0$ for $\kb \in \Lambda \setminus \{ \boldsymbol 0 \}$, and hence $r_\nollb=\langle r,e\rangle$.

In deriving the dual functional $$\varphi(q)=\sup_{\Phi\geq 0,\derivd\nu\geq 0,r}\mathcal{L}(\Phi,\derivd\nu, r, q),$$ to be minimized, we only need to consider $q\in\PPC\setminus\{0\}$, as $\varphi$ will take infinite values for $q\not\in\PPC$. In fact, following along the lines of \cite[p. 1957]{ringh2015multidimensional}, we note that, if $Q(e^{i\thetab_0})<0$, \eqref{eq:Lagrangian} will tend to infinity when $\nu(\thetab_0)\to\infty$. Moreover, since $p\in\PPC\setminus\{0\}$, there is a neighborhood where $P(e^{i\thetab})>0$, letting $\Phi$ tend to infinity in this neighborhood, \eqref{eq:Lagrangian} will tend to infinity if $Q\equiv 0$. We also note that the nonnegative function $\Phi$ can only be zero on a set of measure zero; otherwise the first term in \eqref{eq:Lagrangian} will be $-\infty$.

The directional derivative\footnote{Formally, the Gateaux differential \cite{luenberger1969optimization}.} of the Lagrangian \eqref{eq:Lagrangian} in any feasible direction $\delta\Phi$, i.e., any direction $\delta\Phi$ such that $\Phi +\varepsilon \delta\Phi \geq 0$ for sufficiencly small $\varepsilon >0$,  is easily seen to be
\begin{displaymath}
\delta\mathcal{L}(\Phi,\derivd\nu, r, q;\delta\Phi)= \int_{\mT^d}\left(\frac{P}{\Phi}-Q\right)\delta\Phi\dm.
\end{displaymath} 
In particular, the direction $\delta\Phi:=\Phi\,\text{sign}(P-Q\Phi)$ is feasible since $(1\pm\varepsilon)\Phi\geq 0$ for $0<\varepsilon < 1$. Therefore, any maximizing $\Phi$ must satisfy $\int_{\mT^d}|P-Q\Phi|\dm\leq 0$ and hence 
\eqref{eq:rat}. Moreover, a maximizing choice of $\derivd\nu$ will require that 
\begin{equation}
\label{eq:intQdnu=0}
\int_{\mT^d} Q\derivd\,\nu=0, 
\end{equation}
as this nonnegative term can be made zero by the simple choice $\derivd\nu\equiv 0$, and consequently \eqref{supp(dnu)} must hold.  Finally, the directional derivative 
\begin{displaymath}
\delta\mathcal{L}(\Phi,\derivd\nu, r, q;\delta r)= \langle \delta r, q-e + W^{-1}(r-c)\rangle
\end{displaymath}
is zero for all $\delta r\in \fC$ if
\begin{equation}
\label{eq:q2r}
r=c+W(q-e).
\end{equation} 
Inserting this together with \eqref{eq:rat} and \eqref{eq:intQdnu=0} into  \eqref{eq:Lagrangian} then yields the dual functional 
\begin{displaymath}
\varphi(q)=\langle c, q\rangle - \int_{\mT^d} P\log Q\,\dm +\frac12\|q-e\|_W^2 +c_0.
\end{displaymath}
Consequently the dual of the (primal) optimization problem \eqref{eq:primal_relax} is equivalent to 
\begin{align}\label{eq:dual_relax}
\min_{q\in\PPC} & \quad \langle c, q \rangle - \int_{\mT^d} P \log Q\, \dm+ \frac{1}{2}\|q - e\|_{W}^2 .
%\\
%&\st  \quad Q \in \PPC, 
\end{align}
%where $e := [e_\kb]_{\kb \in \Lambda}$, $e_{\boldsymbol 0} = 1$ and $e_{\kb} = 0$ for $\kb \in \Lambda \setminus \{ \boldsymbol 0 \}$.

\begin{theorem}\label{thm:softcontraints}
For every $p \in \PPC \setminus \{0\}$ the functional in \eqref{eq:dual_relax} is strictly convex and has a unique minimizer $\hat{q}\in \PPC \setminus \{0\}$. Moreover, there exists a unique $\hat{r}\in\CP$, a unique $\hat{c} \in \partial \CP$ and a (not necessarily unique) nonnegative singular measure $\derivd\hat\nu$ with support 
\begin{equation}
\label{supp(dnu)soft}
\supp(\derivd\hat\nu) \subseteq \{ \thetab \in \mT^d  \mid \hat{Q}(e^{i\thetab}) = 0 \}
\end{equation}
such that
\begin{subequations}\label{rhat+chat}
\begin{align}
& \hat{r}_{\kb} = \int_{\mathbb{T}^d} \expfunk \left( \frac{P}{\hat Q}  \dm  + \derivd\hat\nu \right)\, \text{ for all } \kb \in \Lambda, \label{rhat}\\
& \hat{c}_\kb = \int_{\mathbb{T}^d} \expfunk \derivd\hat\nu, \text{ for all } \kb \in \Lambda \, ,\ \label{chat}
\end{align}
\end{subequations}
and the measure 
\begin{equation}
\label{eq:optdmu(soft)}
\derivd\hat\mu(\thetab) = \frac{P(e^{i\thetab})}{\hat{Q}(e^{i\thetab})}\dm(\thetab) + \derivd\hat\nu(\thetab)
\end{equation}
is an optimal solution to the primal problem \eqref{eq:primal_relax}. Moreover, $\derivd\hat\nu$ can be chosen with support in at most $|\Lambda|-1$ points.
\end{theorem}

\begin{proof}
The objective functional $\mJ$ of the dual problem \eqref{eq:dual_relax} can be written as the sum of two terms, namely 
\begin{displaymath}
\mJ_1(q) =  \langle \tilde{c}, q \rangle - \int_{\mT^d} P \log(Q) \dm\quad\text{and}\quad \mJ_2(q)=\langle c-\tilde{c}, q \rangle +\frac{1}{2}\|q - e\|_{W}^2\, ,
\end{displaymath}
%\begin{align*}
%\mJ_1(q) &=  \langle c, q \rangle - \int_{\mT^d} P \log(Q) \dm  \\
%\mJ_2(q) &= \frac{1}{2}\|q - e\|_{W}^2
%\end{align*}
where $\tilde{c}\in\CP$.
The functional $\mJ_1$ is strictly convex (Theorem~\ref{thm:rational}), and trivially the same holds for $\mJ_2$ since it is a positive definite quadratic form. Consequently, $\mJ=\mJ_1+\mJ_2$ is strictly convex, as claimed. Moreover, $\mJ_1$ is lower semicontinuous  \cite[Lemma 3.1]{ringh2015multidimensional} with compact sublevel sets $\mJ_1^{-1}(-\infty,\rho]$ \cite[Lemma 3.2]{ringh2015multidimensional}. Likewise,  $\mJ_2$ is continuous with compact sublevel sets. Therefore $\mJ$ is lower semicontinuous with compact sublevel sets and therefore has a minimum $\hat{q}$, which must be unique by strict convexity. 

In view of \eqref{eq:q2r}, the optimal value of $r$ is given by
\begin{equation}
\label{eq:rhat}
\hat{r}=c +W(\hat{q}-e)
\end{equation}
and is hence unique. Since therefore the linear term $c+W(q-e)$  in the gradient of $\mJ$ takes the value $\hat{r}$ at the optimal point, the analysis in \cite[sect. 3.1.5]{ringh2015multidimensional} applies with obvious modifications, showing that there is a $\hat{c}\in \CPC$, which then must be unique, such that 
\begin{displaymath}
 \hat{r}_{\kb} = \int_{\mathbb{T}^d} \expfunk\frac{P}{\hat Q}  \dm +\hat{c}_{\kb}.
\end{displaymath}
Moreover, there is a discrete measure $\derivd\hat\nu$ with support in at most $|\Lambda|-1$ points such that \eqref{chat} holds; see, e.g., \cite[Proposition 2.4]{karlsson2015themultidimensional}. Then \eqref{rhat} holds as well. In view of \eqref{eq:intQdnu=0}, 
\begin{equation}
\label{eq:compslacknes}
\langle\hat{c},\hat{q}\rangle =\int_{\mathbb{T}^d} \hat{Q}\derivd\hat\nu =0,
\end{equation}
and consequently $\hat{c}\in\partial\CP$, and the support of $\derivd\hat\nu$ must satisfy \eqref{supp(dnu)soft}.

%and a positive measure $\derivd\hat\nu$ such that \eqref{chat} holds. Consequently, \eqref{rhat} holds, and  $\hat{c}$ is unique as claimed. In view of \eqref{eq:intQdnu=0}, 
%\begin{equation}
%%\label{eq:compslacknes}
%\langle\hat{c},\hat{q}\rangle =\int_{\mathbb{T}^d} \hat{Q}\derivd\hat\nu =0,
%\end{equation}
%which is consistent with the fact that $\hat{c}\in\partial\CP$ and shows that \eqref{supp(dnu)soft} must hold. 

Finally, let $r$ be given in terms of $\dmu$ by \eqref{eq:r}, and let $\mI(\dmu)$ be the corresponding primal functional in \eqref{eq:primal_relax}.  Then, for any such $\dmu$, 
\begin{displaymath}
\mI(\dmu)= \mathcal{L}(\Phi,\derivd\nu, r, \hat{q})\leq \mathcal{L}(\hat{\Phi},\derivd\hat\nu, \hat{r}, \hat{q}) =\mI(\derivd\hat\mu),
\end{displaymath}
and hence $\dmuhat$ is an optimal solution to the primal problem \eqref{eq:primal_relax}, as claimed. 
\end{proof}

We collect the KKT conditions in the following corollary. 

\begin{corollary}\label{cor:KKTsoft}
The conditions
\begin{subequations}\label{eq:opt_cond_relax}
\begin{align}
& \hat{q} \in \PPC, \quad \hat c \in \partial \CP, \quad \langle \hat{c}, \hat{q} \rangle = 0 \label{eq:opt_cond_relax_slack} \\
& \hat{r}_\kb = \int_{\mT^d} \expfunk \frac{P}{\hat Q}\dm + \hat{c}_\kb,\quad \kb\in\Lambda \label{eq:opt_cond_relax_KKT} \\
&\hat{r}- c = W(\hat q - e). \label{eq:opt_cond_relax_matchError}
\end{align}
\end{subequations}
are necessary and sufficient conditions for optimality of the dual pair  \eqref{eq:primal_relax} and  \eqref{eq:dual_relax} of optimization problems. 
\end{corollary}

\section{On the well-posedness of the soft-constrained problem}\label{sec:wellposed}

In the previous sections we have shown that the primal and dual optimization problems are well-defined. Next we investigate the well-posedness of the primal problem  as an inverse problem. Thus, we first  establish continuity of the solutions $\hat{q} $ in terms of the parameters $W$, $c$ and $p$.

%If we can show that $c\to(\hat q, \hat c)$  is continuous in $\CP$

\newcommand{\ca}{c}
\newcommand{\cb}{c^{(0)}}
\newcommand{\cc}{c^{(1)}}
\newcommand{\cd}{c^{(2)}}
\newcommand{\ck}{c^{(k)}}
\newcommand{\pa}{p}
\newcommand{\pb}{p^{(0)}}
\newcommand{\pc}{p^{(1)}}
\newcommand{\pd}{p^{(2)}}
\newcommand{\pk}{p^{(k)}}
\newcommand{\Pa}{P}
\newcommand{\Pb}{P^{(0)}}
\newcommand{\Pc}{P^{(1)}}
\newcommand{\Pd}{P^{(2)}}
\newcommand{\Pk}{P^{(k)}}

\newcommand{\Wa}{W}
\newcommand{\Wb}{W^{(0)}}
\newcommand{\Wc}{W^{(1)}}
\newcommand{\Wd}{W^{(2)}}
\newcommand{\Wk}{W^{(k)}}

\newcommand{\PARa}{{\ca, \pa, \Wa}}
\newcommand{\PARb}{{\cb, \pb, \Wb}}
\newcommand{\PARc}{{\cc, \pc, \Wc}}
\newcommand{\PARd}{{\cd, \pd, \Wd}}
\newcommand{\PARk}{{\ck, \pk, \Wk}}

\subsection{Continuity of $\hat{q}$ with respect to $c$, $p$ and $W$} 

We start considering the continuity of the optimal solution with respect to the parameters. The parameter set of interest is
\begin{equation}
\label{eq:Pcal}
\mathcal{P}=\{(\PARa)\mid c\in \fC, p\in \PPC\setminus\{0\}, W>0\}.
\end{equation}

\begin{theorem}\label{thm:W2qcont} 
Let 
\begin{equation}
\label{eq:Jsoft}
\mJ_{\PARa}(q)=\langle c, q \rangle - \int_{\mT^d} P \log Q\, \dm+ \frac{1}{2}\|q - e\|_{W}^2.
\end{equation}
Then the map $(\PARa)\mapsto \hat{q}:=\argmin_{q\in\PPC} \mJ_{\PARa}(q)$ is continuous on  $\mathcal{P}$.
\end{theorem}

\begin{proof}
Following the procedure in  \cite[Proposition  7.3]{karlsson2015themultidimensional} we use the continuity of the optimal value (Lemma~\ref{lem:W2Jcont}) to show continuity of the optimal solution. To this end, let $(\PARk)$ be a sequence of parameters in $\mathcal{P}$ converging to $(\PARa)\in\mathcal{P}$ as $k\to\infty$. Moreover, defining $\mJ_k(q):=\mJ_\PARk (q)$ and $\mJ(q):=\mJ_\PARa (q)$  for simplicity of notation, let $\hat{q}_k = \argmin_{q \in \bar{\mathfrak{P}}_+} \mathbb{J}_{k}(q)$ and $\hat{q} = \argmin_{q \in \bar{\mathfrak{P}}_+} \mathbb{J}(q)$. By Lemma~\ref{lem:W2Jcont}, $(\hat{q}_k)$ is bounded, and hence there is a subsequence, which for simplicity we also call $(\hat{q}_k)$, converging to a limit  $q_\infty$. If we can show that $q_\infty = \hat{q}$, then the theorem follows.  To this end,  choosing a $q_0 \in \mathfrak{P}_+$, we have
\begin{align*}
\mathbb{J}_{k}(\hat{q}_k) &= \mathbb{J}_{k} (\hat{q}_k + \varepsilon q_0) - \langle \ck, \varepsilon q_0 \rangle + \int_{\mT^d} \Pk \log \left( \frac{\hat{Q}_k + \varepsilon Q_0}{\hat{Q}_k} \right)\dm \\ 
&\quad+ \frac12\|\hat{q}_k-e\|_{\Wk}^2 -\frac12 \|\hat{q}_k+\varepsilon q_0-e\|_{\Wk}^2 \\
&\geq \mathbb{J}_{k} (\hat{q}_k + \varepsilon q_0) - \langle \ck, \varepsilon q_0 \rangle + \frac12\|\hat{q}_k-e\|_{\Wk}^2-\frac12 \|\hat{q}_k+\varepsilon q_0-e\|_{\Wk}^2.
\end{align*}
Consequently, by Lemma~\ref{lem:W2Jcont}, 
\begin{equation*}
\mathbb{J}(\hat{q}) = \lim_{k \to \infty} \mathbb{J}_{k}(\hat{q}_k) \geq \lim_{k \to \infty} \mathbb{J}_{k}(\hat{q}_k + \varepsilon q_0) - \varepsilon \langle \ck, q_0 \rangle + \frac12\|\hat{q}_k-e\|_{\Wk}^2 -\frac12 \|\hat{q}_k+\varepsilon q_0-e\|_{\Wk}^2.
\end{equation*}
However $\hat{q}_k + \varepsilon q_0 \in \mathfrak{P}_+$, and, since $(\PARa,q)\mapsto\mathbb{J}_{\PARa}(q)$ is continuous in $\mathcal{P}\times\PP$, we obtain
%\begin{equation}\label{Jpestimate}
%\begin{split}
\begin{align}
\mathbb{J}(\hat{q}) &\geq \lim_{k \to \infty} \left(\mathbb{J}_{k}(\hat{q}_k + \varepsilon q_0) - \varepsilon \langle \ck, q_0 \rangle + \frac12\|\hat{q}_k-e\|_{\Wk}^2-\frac12 \|\hat{q}_k+\varepsilon q_0-e\|_{\Wk}^2\right)\notag\\ 
&= \mathbb{J}(q_\infty + \varepsilon q_0) - \varepsilon \langle c, q_0 \rangle+ \frac12\|q_\infty-e\|_{W}^2 -\frac12 \|q_\infty+\varepsilon q_0-e\|_{W}^2 .\label{Jpestimate}
\end{align}
%\end{split}
%\end{equation}
Letting $\varepsilon\to 0$ in \eqref{Jpestimate}, we obtain the inequality $\mathbb{J}(\hat{q})\geq \mathbb{J}(q_\infty)$. By strict convexity of $\mJ$ the optimal solution is unique, and hence $\hat q=q_\infty$. 
\end{proof}

\subsection{Continuity of $\hat{c}$ with respect to $\hat{q}$} \label{ssec:qhat2chat}

We have now established continuity from $(\PARa)$ to $\hat{q}$. In the same way as in Proposition~\ref{prop:c2qhatchat2} we are also interested in continuity of the map $(\PARa)\mapsto (\hat{q},\hat{c})$. This would follow if we could show that the map from $\hat{q}$ to $\hat{c}$ is continuous. From the KKT condition \eqref{eq:opt_cond_relax_matchError}, it is seen that $\hat{r}$ is continuous in $c$, $W$ and $\hat{q}$. In view of \eqref{eq:opt_cond_relax_KKT}, i.e.,  
\begin{displaymath}
\hat{r}_\kb = \int_{\mT^d} \expfunk \frac{P}{\hat Q}\dm + \hat{c}_\kb,\quad \kb\in\Lambda
\end{displaymath}
continuity of $\hat c$ would  follow if $\int_{\mT^d} P\hat{Q}^{-1}\dm$ is continuous in $(p,\hat{q})$ whenever it is finite. If $p\in\PP$, this follows from the continuity the map $\hat{q}\mapsto \hat{Q}^{-1}$ in $L_1(\mT^d)$.
For the case $d\leq 2$, this is trivial since if  $\int_{\mT^d} \hat{Q}^{-1}\dm$ is finite, then $\hat q\in \PP$ and $\hat Q$ is bounded away from zero (cf., Proposition~\ref{prop:c2qhatchat2}).  However, for the case  $d>2$ the optimal $\hat q$ may belong to the boundary $\partial \PP$, i.e., $\hat Q$ is zero in some point.  The following proposition shows $L_1$ continuity of $\hat{q}\mapsto \hat{Q}^{-1}$  for certain cases.

\begin{proposition} \label{prp:qQinvLt}
For $d\geq3$, let $\hat q\in\PPC$ and suppose that the Hessian $ \nabla_{\thetab\thetab} \, \hat Q$ is positive definite in each point  where $\hat Q$ is zero. Then $\hat Q^{-1}\in L_1(\mT^{d})$ and the mapping from the coefficient vector $q\in \PPC$ to $Q^{-1}$ is $L_1$ continuous in the point $\hat q$.
%For $d=3$, let $\hat q\in\PPC$ and suppose that the Hessian $ \nabla_{\thetab\thetab} \, \hat Q$ is positive definite in each point  where $\hat Q$ is zero. Then $\hat Q^{-1}\in L_1(\mT^3)$ and the mapping from the coefficient vector $q\in \PPC$ to $Q^{-1}$ is $L_1$ continuous in the point $\hat q$.
\end{proposition}

The proof of this proposition is given in \cite{ringh2017further}.  From Propositions~\ref{prp:qQinvLt} and \ref{prop:c2qhatchat2} the following continuity result follows directly.

\begin{corollary}
For all $c\in \fC, p\in\PP, W>0$, the mapping $(c,p,W)\to(\hat q, \hat c)$ is continuous in any point $(\PARa)$ for which the Hessian $ \nabla_{\thetab\thetab} \, \hat Q$ is positive definite in each point where $\hat Q$ is zero.
%\rv{Let $c$,  $p$, $\hat{q}$, and $\hat c$ be as in Theorem~\ref{thm:main}. For all $c\in \CP$ and all $p\in\PP$, the mapping $(c,p) \to(\hat q, \hat c)$ is continuous in any point $(c,p)$  in which $\hat Q$ is strictly positive or in  which the Hessian $ \nabla_{\thetab\thetab} \, \hat Q$ is positive definite in each point where $\hat Q$ is zero.}
%For $d\leq 3$ and all $c\in \fC, p\in\PP, W>0$, the mapping $(c,p,W)\to(\hat q, \hat c)$ is continuous in any point $(\PARa)$ for which the Hessian $ \nabla_{\thetab\thetab} \, \hat Q$ is positive definite in each point where $\hat Q$ is zero.
\end{corollary}

The condition  $p \in \PP$ is needed, since we may have pole-zero cancelations in $P/\hat{Q}$ when $p\in\partial\PP$, and then $\int_{\mT^d} P/\hat{Q}\dm$ may be finite even if $\hat{Q}^{-1}\not\in L_1$. The following example shows that this may lead to discontinuities in the map $p\mapsto \hat{c}$ (cf. Example 3.8 in \cite{karlsson2015themultidimensional}).

\begin{example}
Let
\begin{displaymath}
c = \begin{bmatrix}1 \\ 3 \\ 1\end{bmatrix} = 
\begin{bmatrix}0 \\ 2 \\ 0\end{bmatrix} +\begin{bmatrix}1 \\ 1 \\ 1\end{bmatrix}= 
\int_{-\pi}^\pi\begin{bmatrix}e^{-i\theta} \\ 1 \\ e^{i\theta}\end{bmatrix} \left( 2 \dm +\derivd\nu_0\right),
\end{displaymath}
where $\dm=d\theta/2\pi$ and $\derivd\nu_0$ is the singular measure $\delta_{0}(\theta)d\theta$ with support in $\theta=0$. Since $\dmu:=2 \dm +\derivd\nu_0$ is positive, $c \in \CPC$. Moreover, since
\[
T_{c} = \begin{bmatrix}
3 & 1\\
1 & 3
\end{bmatrix}
> 0
\]
we have that $c \in \CP$ (see, e.g., \cite[p. 2853]{lindquist2013thecirculant}). Thus we know \cite[Corollary 2.3]{ringh2015multidimensional} that for each $p \in \PP$ we have a unique $\hat{q} \in \PP$ such that $P/ \hat{Q}$ matches $c$, and hence $\hat{c} = 0$.  However, for $p = 2 (-1, 2, -1)'$ we have that $\hat{q} = (-1, 2, -1)'$ and $\hat{c} = (1, 1, 1)'$ (Theorem~\ref{thm:rational}). Then, for the sequence $(p_k)$, where $p_k=2 (-1, 2 + 1/k, -1) \in \PP$, we have $\hat{c}_k = 0$, so 
\[
\lim_{k \to \infty} \hat{c}_k = \lim_{k \to \infty}
\begin{bmatrix}
0 \\ 0 \\ 0
\end{bmatrix} \neq 
\begin{bmatrix}
1 \\ 1 \\ 1
\end{bmatrix},
\]
which shows that the mapping $p \to \hat{c}$ is not continuous.
\end{example}

\section{Tuning to avoid a singular part}\label{sec:no_singular_part}

In many situations we prefer solutions where there is no singular measure $d\nu$ in the optimal solution.  An interesting question is therefore for what prior $P$ and  weight $W$ we obtain $\derivd\hat\nu=0$. The following result provides a sufficient condition. 

\begin{proposition}\label{boundlprop}
Let $c\in\fC$  and let $p$ be the Fourier coefficients of the prior $P$. If the weight satisfies\footnote{Here $\|A\|_{2,1}=\max_{c\neq 0}\|Ac\|_1/\|c\|_2$ denotes the subordinate (induced) matrix norm.} 
\begin{equation}\label{eq:supnormbound}
\|W^{-1/2}\|_{2,1}<\|c-p\|_{W^{-1}}^{-1},
\end{equation}
 then the optimal solution of \eqref{eq:primal_relax} is on the form  $$\derivd\hat{\mu}=(P/\hat{Q}) dm,$$ i.e., the singular part $\derivd\hat{\nu}$ vanishes.
\end{proposition}

\begin{remark}
Note that for a scalar weight, $W=\lambda I$ the bound \eqref{eq:supnormbound} simplifies to 
\begin{equation}
\label{eq:scalarbound}
\lambda> |\Lambda|^{1/2}\|c-p\|_2,
\end{equation}
 where $|\Lambda|$ is the cardinality of index set $\Lambda$.
\end{remark}

For the proof of Proposition~\ref{boundlprop} we need the following lemma.

\begin{lemma}\label{boundlemma}
Condition \eqref{eq:supnormbound} implies 
\begin{equation}\label{eq:newnormbound}
\|W^{-1}(\hat{r}-c)\|_1<1,
\end{equation}
where $\hat{r}$ is the optimal value of $r$ in problem \eqref{eq:primal_relax}.
\end{lemma}

\begin{proof}
Let 
\begin{equation}
\label{eq:primalcost}
\mathbb{I}(\dmu,r):=\kl (P\dm, \dmu) + \frac{1}{2}\|r - c\|_{W^{-1}}^2 
\end{equation}
be the cost function of problem \eqref{eq:primal_relax}, and let $(\derivd\hat{\mu},\hat{r})$ be the optimal solution. Clearly, $\mathbb{I}(Pdm,p)\geq \mathbb{I}(\derivd\hat{\mu},\hat{r})$, and consequently 
\begin{displaymath}
\|\hat{r}-c\|_{W^{-1}} \leq \|p-c\|_{W^{-1}},
\end{displaymath}
since $\kl (P\dm, \derivd\hat{\mu})\geq 0$ and $\kl (P\dm,P\dm)=0$. Therefore,
\begin{align*}
\|W^{-1}(\hat{r} - c)\|_1
&\le \|W^{-1/2}\|_{2,1} \|W^{-1/2}(\hat{r} - c)\|_2\\
&= \|W^{-1/2}\|_{2,1}\|\hat{r} - c\|_{W^{-1}}\\
&\le \|W^{-1/2}\|_{2,1}  \|p - c\|_{W^{-1}},
\end{align*}
 which is less than one by \eqref{eq:supnormbound}. Hence  \eqref{eq:supnormbound} implies \eqref{eq:newnormbound}.
\end{proof}

\begin{proof}[{Proof of Proposition~\ref{boundlprop}}]
%{\em Proof of Proposition~\ref{boundlprop}:}
%\begin{proofWithName}
Suppose the optimal solution has a nonzero singular part $\derivd\hat{\nu}$, and form the  directional derivative of \eqref{eq:primalcost} at $(\derivd\hat{\mu},\hat{r})$ in the direction $-\derivd\hat{\nu}$. Then $\Phi$ in \eqref{eq:decomp} does not vary, and 
\begin{displaymath}
\delta\mathbb{I}(\derivd\hat{\mu},\hat{r};-\derivd\hat{\nu},\delta r)=-\int_{\mT^d}\derivd\hat{\nu} +\delta r^*W^{-1}(\hat{r} - c),
\end{displaymath}
where
\begin{displaymath}
\delta r_\kb = -\int_{\mT^d} \expfunk\derivd\hat{\nu}.
\end{displaymath}
Then $|\delta r_\kb|\leq \int \derivd\hat{\nu}$ for all $\kb\in\Lambda$, and hence
\begin{displaymath}
|\delta r^*W^{-1}(\hat{r} - c)|\leq \|W^{-1}(\hat{r}-c)\|_1 \int_{\mT^d} \derivd\hat{\nu} <\int_{\mT^d} \derivd\hat{\nu},
\end{displaymath}
by \eqref{eq:newnormbound} (Lemma~\ref{boundlemma}). Consequently,
\begin{displaymath}
\delta\mathbb{I}(\derivd\hat{\mu},\hat{r};-\derivd\hat{\nu},\delta r)<0
\end{displaymath}
whenever $\derivd\hat{\nu}\ne 0$, which contradicts optimality. Hence $\derivd\hat{\nu}$ must be zero. 
%\rule{1.4ex}{1.4ex}
%\end{proofWithName}
\end{proof}

The condition of Proposition~\ref{boundlprop} is just sufficient and is in general conservative. To illustrate this, we consider a simple one-dimensional example ($d=1$). 

\begin{example}\label{ex:singular}
Consider a covariance sequence $(1,c_1)$, where $c_1\ne 0$, and a prior $P(e^{i\theta})=1-\cos\theta$, and set  $W =\lambda I$. Then, since 
\begin{displaymath}
c=\begin{pmatrix} c_1\\ 1 \\c_1 \end{pmatrix} \quad \text{and} \quad p=\begin{pmatrix} -1/2\\ 1 \\-1/2 \end{pmatrix},
\end{displaymath}
the sufficient condition \eqref{eq:scalarbound} for an absolutely continuous solution is 
\begin{equation}
\label{bound}
\lambda > \sqrt{\tfrac32}\, |1+2c_1|.
\end{equation}
We want to investigate how restrictive this condition is. 

Clearly we will have a singular part  if and only if $\hat{Q} = q_0 P$, in which case we have 
\begin{displaymath}
\hat{q}=q_0\begin{pmatrix} -1/2\\ 1 \\-1/2 \end{pmatrix}\quad \text{and} \quad\hat c=\beta\begin{pmatrix} 1\\ 1 \\1
\end{pmatrix}
\end{displaymath}
for some $\beta >0$. In fact, it follows from $\langle \hat{c}, \hat{q} \rangle = 0$ in \eqref{eq:opt_cond_relax_slack} that $\hat{c}_1=\hat{c}_0$. Moreover, \eqref{eq:opt_cond_relax_KKT} and  \eqref{eq:opt_cond_relax_matchError} yield
\begin{align*}
\hat{r} =&\int \frac{P}{\hat{Q}}\begin{pmatrix} e^{i\theta}\\ 1 \\e^{-i\theta}
\end{pmatrix} dm +\hat c=
\begin{pmatrix} \beta\\ \beta+1/q_0 \\\beta
\end{pmatrix}\\
c = &\,\hat{r}-\lambda(q-e)=\begin{pmatrix} \beta+\lambda q_0/2,\\ \beta+1/q_0-\lambda  q_0+\lambda  \\\beta+\lambda  q_0/2
\end{pmatrix}.
\end{align*}
By eliminating $\beta$, we get 
\[
c_1=1-\frac{1}{q_0}-\frac{3}{2}q_0 \lambda  +\lambda,
\]
and solving for $q_0$ yields
\[
q_0=\frac{\lambda  +c_1-1+ (6\lambda  + (\lambda  +c_1-1)^2 )^{1/2}}{3\lambda }
\]
(note that $\lambda >0$ and $q_0>0$). Again, using \eqref{eq:opt_cond_relax_matchError} we have
\begin{align*}
\beta&=c_1-\lambda  q_0/2\\
&=c_1-\frac{1}{6}\left( \lambda  +c_1-1+ (6\lambda  + (\lambda  +c_1-1)^2)^{1/2}\right).
\end{align*}
We are interested in $\lambda$ for which $\beta >0$, i.e.,
\begin{equation}
\label{eq:inex}
6c_1-(\lambda  +c_1-1)>(6\lambda  + (\lambda  +c_1-1)^2)^{1/2},
\end{equation}
which is equivalent to the two conditions
\begin{subequations}\label{twoconditions}
\begin{equation}
1+5c_1>\lambda 
\end{equation}
\begin{equation}
2c_1(1+2c_1)>\lambda  (1+2c_1),
\end{equation}
\end{subequations}
which could be seen by noting that the left member of \eqref{eq:inex} must be positive and then squaring both sides.
To find out whether this has a solution we consider three cases, namely $c_1<-1/2$, $-1/2<c_1<0$, and $c_1>0$. For $c_1<-1/2$, condition \eqref{twoconditions} becomes 
$2c_1<\lambda <1+5c_1$,
which is impossible since $1+5c_1<2c_1$. Condition \eqref{twoconditions} cannot be satisfied when $-1/2<c_1<0$, because then $\lambda$ would be negative which contradicts $\lambda >0$. When $c_1>0$,  Condition \eqref{twoconditions} is satisfied if and only if $\lambda <2c_1$.

Consequently, there is no singular part if either $c_1$ is negative or $$\lambda\geq 2c_1.$$This shows that the condition \eqref{bound} is not tight.
\end{example}

\section{Covariance extension with hard constraints}\label{sec:approxhard}
The alternative optimization problem  \eqref{eq:primal_relax_new} amounts to minimizing $\kl (P\dm, \dmu)$ subject to the hard constraint  $\|r - c \|_{W^{-1}}^2 \le 1$, where $r_\kb = \int_{\mT^d} \expfunk \dmu$. Hard constraints of this type were used in \cite{schott1984maximum} in the context of entropy maximization. In general the data $c\not\in\CPC$, whereas, by definition, $r\in\CPC$. Consequently, a necessary condition for the existence of a solution is that $\CPC$ and the strictly convex set
\begin{equation}
\label{eq:SW}
\SW=\{ r\mid \|r - c \|_{W^{-1}}^2 \le 1\}
\end{equation}
have a nonempty intersection. In the case that $\SW\cap\CPC\subset \partial\CP$, this intersection only contains one point \cite[Section 3.12]{luenberger1969optimization}. In this case, any solution to the moment problem contains only a singular part (Proposition~\ref{prop:singular}), and then the primal problem \eqref{eq:primal_relax_new} has a unique feasible point $r$, but the objective function is infinite. Moreover, $\kl (P\dm, \dmu)\geq 0$ is strictly convex with $\kl (P\dm, P\dm)= 0$, so if $p\in\SW$ then \eqref{eq:primal_relax_new} has the trivial unique optimal solution $\dmuhat =P\dm$, and $\hat{r}=p$. The remaining case, $p\not\in\SW\cap\CP\ne\emptyset$ needs further analysis. 

To this end, setting $\dmu=\Phi\dm+\derivd\nu$, we consider the Lagrangian 
\begin{align*}
\mathcal{L}(\Phi,\derivd\nu, r, q, \gamma) = & \; -\kl (P\dm, \dmu)
+ \sum_{\kb \in \Lambda} q_\kb^* \left( r_\kb - \int_{\mT^d} \expfunk \dmu(\thetab) \right) \\
& \; +\gamma \left(1- \|r - c \|_{W^{-1}}^2 \right)\\
 = & \; -\kl (P\dm, \dmu) + \langle r,q\rangle - \int_{\mT^d} Q \dmu +\gamma \left(1- \|r - c \|_{W^{-1}}^2 \right), 
\end{align*}
where $\gamma\ge 0$. Therefore, in view of \eqref{D(Pdm,dmu)}, 
\begin{align}\label{eq:Lagrangian2}
\mathcal{L}(\Phi,\derivd\nu, r, q,\gamma) = &\; \int_{\mT^d}P\log\Phi dm -  \int_{\mT^d}Q\Phi dm -  \int_{\mT^d} Q\derivd\nu -\int_{\mT^d}P(\log P -1)dm \notag \\
&\; +\langle r,q-e\rangle +\gamma \left(1-\|r - c \|_{W^{-1}}^2 \right) , 
\end{align}
where, as before,  $e := [e_\kb]_{\kb \in \Lambda}$, $e_{\boldsymbol 0} = 1$ and $e_{\kb} = 0$ for $\kb \in \Lambda \setminus \{ \boldsymbol 0 \}$, and hence $r_\nollb=\langle r,e\rangle$. This Lagrangian differs from that in \eqref{eq:Lagrangian} only in the last term that does not depend on $\Phi$. Therefore, in deriving the dual functional $$\varphi(q,\gamma)=\sup_{\Phi\geq 0,\derivd\nu\geq 0,r}\mathcal{L}(\Phi,\derivd\nu, r, q,\gamma),$$  we only need to consider $q\in\PPC\setminus\{0\}$, and a first variation in $\Phi$ yields \eqref{eq:rat} and  \eqref{eq:intQdnu=0}. The directional derivative 
\begin{displaymath}
\delta\mathcal{L}(\Phi,\derivd\nu, r, q,\gamma;\delta r)= q-e + 2\gamma W^{-1}(r-c)
\end{displaymath}
is zero for 
\begin{equation}
\label{eq:q2rhard}
r=c+\frac{1}{2\gamma}W(q-e).
\end{equation} 
Thus inserting \eqref{eq:rat} and \eqref{eq:intQdnu=0} and \eqref{eq:q2rhard} into  \eqref{eq:Lagrangian2} yields the dual functional
\begin{equation}
\label{eq:harddual}
\varphi(q,\gamma)=\langle c, q\rangle - \int_{\mT^d} P\log Q\,\dm +\frac{1}{4\gamma}\|q-e\|_W^2 +\gamma -c_\nollb
\end{equation}
to be minimized over all $q\in\PPC\setminus\{0\}$ and $\gamma\geq 0$. Since $\frac{d\varphi}{d\gamma}= -\frac{1}{4\gamma^2}\|q-e\|_W^2 +1$, there is a stationary point
\begin{equation}
\label{eq:gammaopt}
\gamma =\frac12\|q-e\|_W
\end{equation}
that is nonnegative as required. 

For $\gamma=0$ we must have $q=e$, and consequently $\varphi(q,\gamma)$ tends to zero as $\gamma\to 0$. By weak duality zero is therefore a lower bound for the minimization problem \eqref{eq:primal_relax_new}, and $\kl (P\dm, \dmuhat)=0$, which corresponds to the trivial unique solution $\dmuhat=P\dm$ and $\hat{r}=p$ mentioned above. This solution is only feasible if $p\in\SW$. 
Therefore we can restrict our attention to the case $\gamma>0$. Inserting \eqref{eq:gammaopt} into \eqref{eq:harddual} and removing the constant term $c_\nollb$, we obtain the modified dual functional
\begin{equation}
\label{eq:moddual}
\mJ(q)=\langle c, q\rangle - \int_{\mT^d} P\log Q\,\dm + \|q-e\|_W.
\end{equation}
Moreover, combining \eqref{eq:q2rhard} and \eqref{eq:gammaopt}, we obtain
\begin{equation}
\label{eq:boundaryr}
\|r-c\|_{W^{-1}}=1\, ,
\end{equation}
which also follows from complementary slackness since $\gamma>0$ and restricts $r$ to the boundary of $\SW$. 

\begin{theorem}\label{thm:hardcontraints}
Suppose that $p\in\PPC \setminus\{0\}$, $p\not\in\SW$ and $\SW\cap\CP\ne\emptyset$. Then the modified dual problem 
\begin{equation}
\label{eq:moddualproblem}
\min_{q\in\PPC}\mJ(q)
\end{equation} 
has a unique solution $\hat{q}\in\PPC \setminus\{0\}$. Moreover, there exists a unique $\hat{r}\in\CP$, a unique $\hat{c} \in \partial \CP$ and a (not necessarily unique) nonnegative singular measure $\derivd\hat\nu$ with support 
\begin{equation}
\label{supp(dnu)hard}
\supp(\derivd\hat\nu) \subseteq \{ \thetab \in \mathbb{T}^d  \mid \hat{Q}(e^{i\thetab}) = 0 \}
\end{equation}
such that
\begin{subequations}\label{rhat2+chat2}
\begin{align}
& \hat{r}_{\kb} = \int_{\mathbb{T}^d} \expfunk \left( \frac{P}{\hat Q}  \dm  + \derivd\hat\nu \right)\, \text{ for all } \kb \in \Lambda, \label{rhat2}\\
& \hat{c}_\kb = \int_{\mathbb{T}^d} \expfunk \derivd\hat\nu, \text{ for all } \kb \in \Lambda \, ,\ \label{chat2}
\end{align}
\end{subequations}
and the measure 
\begin{equation}
\label{eq:optdmu(hard)}
\derivd\hat\mu(\thetab) = \frac{P(e^{i\thetab})}{\hat{Q}(e^{i\thetab})}\dm(\thetab) + \derivd\hat\nu(\thetab)
\end{equation}
is an optimal solution to the primal problem \eqref{eq:primal_relax_new}. Moreover, 
\begin{equation}
\label{eq:rhat&c}
\|\hat{r}-c\|_{W^{-1}}=1\, ,
\end{equation}
and $\derivd\hat\nu$ can be chosen with support in at most $|\Lambda|-1$ points. 

If $p\in\SW$, the unique optimal solution is $\dmuhat=P\dm$, and then $\hat{r}=p$. If $\SW\cap\CPC\subset \partial\CP$, any solution to the moment problem will have only a singular part.  Finally, if $\SW\cap\CPC=\emptyset$, then the problem \eqref{eq:primal_relax_new} will have no solution. 
\end{theorem}

\begin{proof}
We begin by showing that the functional $\mJ$ has a minimum under the stated conditions. To this end,  
we first establish that the functional $\mJ$ has compact sublevel sets $\mJ^{-1}(-\infty,\rho]$, i.e., $\|q\|_\infty$ is bounded for all $q$ such  that  $\mJ(q)\leq\rho$, where $\rho$ is sufficiently large for the sublevel set to be nonempty. The functional \eqref{eq:moddual} can be decomposed in a linear and a logarithmic term as
\begin{displaymath}
\mJ(q)=h(q) - \int_{\mT^d} P\log Q\,\dm +c_\nollb\, ,
\end{displaymath}
where $h(q):= \langle c, q-e\rangle +\|q-e\|_W$. The integral term will tend to $-\infty$ as $\|q\|_\infty\to \infty$. 
Therefore we need to have the linear term to tend to $+\infty$ as $\|q\|_\infty\to \infty$, in which case we can appeal to the fact that linear growth is faster than logarithmic growth.  However, if $c\not\in\CPC$ as is generally assumed, there is a $q\in\PPC$ such that  $\langle c, q\rangle< 0$, so we need to ensure that the positive term $\|q-e\|_W$ dominates. 

Let $\tilde{r}\in\SW\cap\CP\ne\emptyset$. Then,  by Theorem~\ref{thm:rational},  there is a positive measure $\derivd\tilde{\mu}=\tilde{\Phi}dm+ \derivd\tilde{\nu}$ with a nonzero $\tilde{\Phi}$ such that 
\begin{displaymath}
\tilde{r}=\int_{\mT^d} \expfunk \derivd\tilde\mu \, ,
\end{displaymath} 
and $\tilde{r}$ satisfies the constraints in the primal problem  \eqref{eq:primal_relax_new}. Consequently,   
\begin{displaymath}
\varphi(q,\gamma)\geq\mathcal{L}(\tilde\Phi,\derivd\tilde\nu, \tilde{r}, q,\gamma)\geq -\mD(P\dm,\derivd\tilde{\mu})
\end{displaymath}
 for all $q\in\PPC$ and $\gamma \geq 0$, which in particular implies that 
\begin{equation}
\label{eq:lowerbound}
\mJ(q)\geq -\mD(P\dm,\derivd\tilde{\mu}) \quad \text{ for all $q\in\PPC$}.
\end{equation}
Now, if there is a $q\in\PPC$ such that $h(q)\leq 0$, then $\mJ(\lambda q)\to -\infty$ as $\lambda\to\infty$, which contradicts \eqref{eq:lowerbound}. Therefore, $h(q)> 0$ for all $q\in\PPC$. Then, since $h$ is continuous, it has  a minimum $\varepsilon$ on the compact set $K:=\{q\in\PPC\setminus\{0\}\mid \|q-e\|_\infty =1\}$. As $e\not\in K$,  $\epsilon >0$. Therefore,
\begin{displaymath}
h(q) \geq  \varepsilon  \|q-e\|_\infty \geq  \varepsilon\|q\|_\infty - \varepsilon\|e\|_\infty \geq \frac{\varepsilon}{|\Lambda |} \|Q\|_\infty  - \varepsilon\|e\|_\infty \, ,
\end{displaymath}
since $\|Q\|_\infty \leq |\Lambda|\|q\|_\infty$ \cite[Lem. A.1]{ringh2015multidimensional}.  Likewise, 
\begin{align*}
%\begin{split}
\int_{\mathbb{T}^d} \! \! P \log Q \dm &=\int_{\mathbb{T}^d} \! \!  P \log \left[ \frac{Q}{\|Q\|_\infty} \right] \dm\!  +\! \int_{\mathbb{T}^d}\! \!  P  \log \|Q\|_\infty \dm\\ & \leq \int_{\mathbb{T}^d}\! \!  P \log \|Q\|_\infty \dm\, ,
%\end{split}
\end{align*}
since $Q/\|Q\|_\infty \leq 1$. Hence
\begin{equation}\label{eq:Jhard_ineq}
\rho \geq \mJ(q) \geq  \frac{\varepsilon}{|\Lambda |} \|Q\|_\infty -\int_{\mathbb{T}^d}\! \!  P \log \|Q\|_\infty \dm  - \varepsilon\|e\|_\infty \, .
\end{equation}
Comparing linear and logarithmic growth  we see that the sublevel set is bounded from above and below.  Moreover, a trivial modification of \cite[Lemma 3.1]{ringh2015multidimensional} shows that $\mJ$ is lower semi-continuous, and hence  $\mJ^{-1}(-\infty,\rho]$ is compact. Consequently, the problem \eqref{eq:moddualproblem} has an optimal solution $\hat{q}$.

%If $W>cc^*$, then $(q-e)^*W(q-e) \ar{\geq} \langle c, q-e\rangle ^2$ \ar{and with equality only for $q=e$}. Hence $\|q-e\|_W \ar{\geq} |\langle c, q-e\rangle|$, i.e., $h(q) \ar{\geq} 0$ for all $q\in\PPC\setminus\{0\}$. Consequently, since $h$ is continuous, it has has a minimum on the compact set $\{q\in\PPC\setminus\{0\}\mid \|q-e\|_\infty =1\}$. \ar{However, since $\{ q = e \}$ does not lie on this compact set, we have that the minimum is} $\varepsilon >0$ and therefore 

Next we show that $\hat{q}$ is unique. For this we return to the original dual problem to find a minimum of \eqref{eq:harddual}. The solution $\hat{q}$ is a minimizer of $\varphi(q,\hat{\gamma})$, where 
\begin{displaymath}
\hat{\gamma}=\frac12\|\hat{q}-e\|_W\, ,
\end{displaymath}
and $\mJ(\hat{q})=\varphi(\hat{q},\hat{\gamma})+c_\nollb$.
To show that $\varphi$ is strictly convex, we form the Hessian
\begin{displaymath}
H=\begin{bmatrix} \int_{\mathbb{T}^d}P/Q^2\dm &0\\0&0 \end{bmatrix}
+\frac{1}{2\gamma^3}\begin{bmatrix} \gamma^2 W& -\gamma (q-e)^*W\\-\gamma W(q-e)&(q-e)^*W(q-e)\end{bmatrix}\
\end{displaymath}
and the quadratic form
\begin{displaymath}
\begin{bmatrix} x\\ \xi\end{bmatrix}^*H\begin{bmatrix} x\\ \xi\end{bmatrix}=
x^* \left(\int_{\mathbb{T}^d}P/Q^2\dm\right) x + \frac{1}{2\gamma^3} [\gamma x - \xi (q-e)]^*W[\gamma x - \xi (q-e)],
\end{displaymath}
which is positive for all nonzero $(x,\xi)$, since $(q-e)\ne 0$ and $\gamma>0$. Consequently, $\varphi$ has a unique minimizer $(\hat{q},\hat{\gamma})$, where $\hat{q}$ is the unique minimizer of $\mJ$.

It follows from \eqref{eq:q2rhard} and \eqref{eq:gammaopt} that 
\begin{equation}
\label{eq:rhathard}
\hat{r}=c + \frac{W(\hat{q}-e)}{\|\hat{q}-e\|_W}\, ,
\end{equation}
which consequently is unique. Moreover, $h(\hat{q})=\langle\hat{r},\hat{q}\rangle -\hat{r}_\nollb$, and hence we can follow the same line of proof as in Theorem~\ref{thm:softcontraints} to show that there is a unique $\hat{c}\in\partial\CP$ such that $\langle\hat{c},\hat{q}\rangle=0$ and a positive discrete measure $\derivd\hat\nu$ with support in $|\Lambda| -1$ points so that \eqref{supp(dnu)hard} and \eqref{rhat2+chat2} hold. Next, let $\mI(\dmu)=-\mD(Pdm,\dmu)$ be the primal functional in \eqref{eq:primal_relax_new}, where $\dmu$ is restricted to the set of positive measures $\dmu:=\Phi\dm + d\nu$  such that $r$, given by \eqref{eq:r}, satisfies the constraint $\|r-c\|_W\leq 1$. 
In view of \eqref{eq:rhat&c}, 
\begin{displaymath}
\mI(\dmu)= \mathcal{L}(\Phi,\derivd\nu, r, \hat{q},\hat{\gamma})\leq \mathcal{L}(\hat{\Phi},\derivd\hat\nu, \hat{r}, \hat{q},\hat{\gamma}) =\mI(\derivd\hat\mu)
\end{displaymath}
for any such $\dmu$, and hence $\dmuhat$ is an optimal solution to the primal problem \eqref{eq:primal_relax_new}. 
Finally, the cases $p\in\SW$, $\SW\cap\CPC\subset \partial\CP$, and $\SW\cap\CPC=\emptyset$ have already been discussed above.
\end{proof}

\begin{corollary}\label{cor:KKThard}
Suppose that $p\in\PPC \setminus\{0\}$ and $\SW\cap\CP\ne\emptyset$. The KKT conditions
\begin{subequations}\label{eq:opt_cond_relax_new}
\begin{align}
& \hat{q} \in \PPC, \quad \hat c \in \partial \CP, \quad \langle \hat{c}, \hat{q} \rangle = 0 \label{eq:opt_cond_relax_slack_new} \\
& \hat{r}_\kb = \int_{\mT^d} \expfunk \frac{P}{\hat Q}\dm + \hat{c}_\kb,\quad \kb\in\Lambda \label{eq:opt_cond_relax_KKT_new} \\
&(\hat{r}-c)\|\hat{q}-e\|_W = W(\hat{q}-e), \quad \hat{r}\in\SW \label{eq:opt_cond_relax_matchError_new}
\end{align}
\end{subequations}
are necessary and sufficient conditions for optimality of the dual pair  \eqref{eq:primal_relax_new} and  \eqref{eq:moddualproblem} of optimization problems. 
\end{corollary}

The corollary follows by noting that, if $p\in\SW$, then we obtain the trivial solution $\hat{q}=e$, which corresponds to the primal optimal solution $\dmuhat =P\dm$. 

\begin{proposition}\label{prop:conditionW}
The condition
\begin{equation}
\label{eq:W>cc*}
W > cc^*
\end{equation}
is sufficient for the pair \eqref{eq:primal_relax_new} and \eqref{eq:moddualproblem} of dual problems to have optimal solutions. 
\end{proposition}

\begin{proof}
If $W>cc^*$, then $(q-e)^*W(q-e)\geq \langle c, q-e\rangle ^2$ with equality only for $q=e$. Hence, if $q\ne e$, $\|q-e\|_W > |\langle c, q-e\rangle|$, i.e., $h(q) > 0$ for all $q\in\PPC\setminus\{0\}$ except $q=e$. Then we proceed as in the proof of Theorem~\ref{thm:hardcontraints}.
\end{proof}

\begin{remark}
Condition \eqref{eq:W>cc*} guarantees that $0 \in {\rm int}(\SW)$ and hence in particular that $\SW\cap\CP\ne\emptyset$ as required in Theorem~\ref{thm:hardcontraints}. To see this, note that $0\in\CPC$ and that $r=0$ satisfies the hard constraint in \eqref{eq:primal_relax_new} if $c^*W^{-1}c\leq 0$. However, since $W>cc^*$, there is a $W_0 > 0$ such that $W = W_0 + cc^*$. Then the well-known Matrix Inversion Lemma (see, e.g., \cite[p. 746]{lindquist2015linear}) yields
\[
(W_0 + cc^*)^{-1} = W_0^{-1} - W_0^{-1} c (1 + c^*W_0^{-1}c)^{-1} c^* W_0^{-1},
\]
and therefore
\[
c^*W^{-1}c = c^*W_0^{-1}c - c^*W_0^{-1} c (1 + c^*W_0^{-1}c)^{-1} c^* W_0^{-1}c = \frac{c^*W_0^{-1}c}{1 + c^*W_0^{-1}c} < 1,
\]
which establishes that $0 \in {\rm int}(\SW)$. However, for $\SW\cap\CP$ to be nonempty, $r=0$ need not be contained in this set. Hence, condition \eqref{eq:W>cc*} is not necessary, although it is easily testable. In fact, this provides an alternative proof of Proposition~\ref{prop:conditionW}.
\end{remark}

\section{On the equivalence between the two problems}\label{sec:homeomorphism} 

Clearly $\SW\cap\CP$ is always nonempty if $c\in\CP$. Then both the problem \eqref{eq:primal_relax} with soft constraints and the problem \eqref{eq:primal_relax_new} with hard constraints  have a solution for any choice of W. On the other hand, if $c \not \in \CP$, the problem with soft constraints will always have a solution, while the problem with hard constraints may fail to have one for certain choices of $W$. However, if the weight matrix in the hard-constrained problem -- let us denote it $W_{\rm hard}$ -- is chosen in the set $\mathcal{W}:=\{W > 0 \mid \SW\cap\CP\ne\emptyset, p\not\in\SW\}$, then it can be seen  from Corollaries \ref{cor:KKTsoft} and \ref{cor:KKThard} that we obtain exactly the same solution $\hat{q}$ in the soft-constrained problem by choosing 
\begin{equation}
\label{eq:Whard2Wsoft}
W_{\rm soft}=W_{\rm hard}/\|\hat{q}-e\|_{W_{\rm hard}}. 
\end{equation}
We note that  \eqref{eq:Whard2Wsoft} can be written $W_{\rm hard}=\alpha W_{\rm soft}$, where $\alpha :=\|\hat{q}-e\|_{W_{\rm hard}}$. Therefore, substituting $W_{\rm hard}$ in \eqref{eq:Whard2Wsoft}, we obtain 
 \begin{displaymath}
W_{\rm soft}=\frac{\alpha W_{\rm soft}}{\|\hat{q}-e\|_{\alpha W_{\rm soft}}}=\alpha^{1/2}\frac{W_{\rm soft}}{\|\hat{q}-e\|_{W_{\rm soft}}},
\end{displaymath}
which yields  $\alpha=\|\hat{q}-e\|_{W_{\rm soft}}^2$. Hence the inverse of \eqref{eq:Whard2Wsoft} is given by 
\begin{equation}
\label{eq:Wsoft2Whard}
W_{\rm hard}=W_{\rm soft}\|\hat{q}-e\|_{W_{\rm soft}}^2. 
\end{equation}
By Theorem~\ref{thm:W2qcont} $\hat{q}$ is continuous in $W_{\rm soft}$, and hence, by \eqref{eq:Wsoft2Whard}, the corresponding $W_{\rm hard}$ varies continuously with $W_{\rm soft}$. In fact, this can be strengthened to a homeomorphism between the two weight matrices. 

\begin{theorem}\label{thm:W2Whom}
The map \eqref{eq:Whard2Wsoft} is a homeomorphism between $\mathcal{W}$ and the space of all (Hermitian positive definite) weight matrices, and the inverse is given by \eqref{eq:Wsoft2Whard}.
\end{theorem}

\begin{proof}
 By \cite[Lemma 2.3]{byrnes2007interior}, a continuous map  between two spaces of the same dimension is a homeomorphism if and only if it is injective and proper, i.e., the preimage  of any compact set is compact. To see that $\mathcal{W}$ is open, we observe that $\SW$ is continuous in $W$ and that $\CP$ is an open set. As noted above, the map \eqref{eq:Wsoft2Whard} -- let us call it $f$ -- is continuous  and  also injective, as it can be inverted. Hence it only remains to show that $f$ is proper. To this end, we take a compact set $K\subset\mathcal{W}$ and  show that $f^{-1}(K)$ is also compact. There are two ways this could fail. First, the preimage could contain a singular semidefinite matrix. However this is impossible by \eqref{eq:Wsoft2Whard}, since $\|\hat{q}\|_\infty$ is bounded for $W_{\text{hard}}\in K$ (Lemma~\ref{lem:qhatbounded}) and  a nonzero scaling of a singular matrix cannot be nonsingular. Secondly, $\|W_{\rm soft}\|_F$ could tend to infinity. However, this is also impossible. To see this, we first show that there is a $\kappa >0$ such that $\|p-r\|_{W_{\rm hard}^{-1}}\geq \kappa$ for all $r\in\mathfrak{S}_{W_{\rm hard}}$ and all $W_{\rm hard}\in K$. To this end, we observe that the minimum of $\|p-r\|_{W^{-1}}$ over all $W\in K$ and $r$ satisfying the constraint $\|r-c\|_{W^{-1}}\leq 1$ is bounded by 
 \begin{displaymath}
\kappa:= \min_{W\in K} \|p-c\|_{W^{-1}} -1
\end{displaymath}
 by the triangle inequality $\|p-r\|_{W^{-1}}\geq\|p-c\|_{W^{-1}}-\|c-r\|_{W^{-1}}\geq\|p-c\|_{W^{-1}}-1$.
 The minimum is attained, since $K$ is compact, and positive, since  $p\not\in\bigcup_{W\in K}\mathfrak{S}_{W}$.
Now, from Corollary~\ref{cor:KKThard} we see that $\hat{q}=e$ if and only if $\hat{r}=p$. The map from $\hat{q}\mapsto\hat{r}$ is continuous in $q=e$. In fact, $\hat Q$ is uniformly positive in a neighborhood of $e$ and hence the corresponding $\hat{c}=0$. 
Due to this continuity, if $\hat{q}\to e$, then $\hat{r}\to p$, which cannot happen since $\|p-r\|_{W^{-1}}\geq \kappa$ for all $W\in K$. Thus, since $\|\hat{q}-e\|_{W}$ is bounded away from zero, the preimage $f^{-1}(K)$ of $K$ is bounded. Finally, consider a convergent sequence $(W_k)$  in $f^{-1}(K)$ converging to a limit $W_\infty$. Since the sequence is bounded and cannot converge to a singular matrix, we must have $W_\infty >0$, i.e.,  $W_\infty\in f^{-1}(\mathcal{W})$. By continuity, $f(W_k)$ tends to the limit $f(W_\infty)$, which must belong to $K$ since it is compact.  Hence the preimage $W_\infty$ must belong to $f^{-1}(K)$.  Therefore, $f^{-1}(K)$ is compact as claimed. 
\end{proof}

It is illustrative to consider the simple case when $W=\lambda I$. Then the two maps \eqref{eq:Whard2Wsoft} and \eqref{eq:Wsoft2Whard} become
\begin{equation}
\label{eq:Lambda2lambda}
\begin{split}
\lambda_{\rm soft}&=\frac{\sqrt{\lambda_{\rm hard}}}{\|\hat{q}-e\|_2}\\
\lambda_{\rm hard}&=\lambda_{\rm soft}^2\|\hat{q}-e\|_2^2
\end{split}
\end{equation}
Whereas the range of $\lambda_{\rm soft}$ is the semi-infinite interval $(0,\infty)$, for the homeomorphism to hold  $\lambda_{\rm hard}$ is confined to 
\begin{displaymath}
\lambda_{\rm min} <\lambda< \lambda_{\rm max},
\end{displaymath}
where $\lambda_{\rm min}$ is the distance from $c$ to the cone $\CPC$ and $\lambda_{\rm max}=\|c-p\|$. When $\lambda_{\rm soft}\to\infty$, $\lambda_{\rm hard}\to\lambda_{\rm max}$ and $\hat{q}\to e$. If $\lambda_{\rm hard}\geq\lambda_{\rm max}$, then the coresponding problem has the trivial unique solution $\hat{q}= e$, corresponding to the primal solution $\dmuhat=P\dm$.

%Theorem~\ref{thm:W2Whom}
%
%\begin{theorem}\label{thm:par2parhom}
%The map \eqref{eq:Whard2Wsoft} is a homeomorphism between $\mathcal{W}$ and the space of all (symmetric positive definite) weight matrices, and the inverse is given by \eqref{eq:Wsoft2Whard}.
%\end{theorem}

%\subsection{Well-posedness  in the hard-constrained problem} 

Note that Theorem~\ref{thm:W2Whom} implies that some continuity results in one of the problems can be automatically transferred to the other problem. In particular, we have the following result.

\begin{theorem}\label{thm:W2qcont_hard}  
Let 
\begin{equation}
\label{eq:Jhard}
\mJ_W(q)=\langle c, q \rangle - \int_{\mT^d} P \log Q\, \dm+ \|q - e\|_{W}.
\end{equation}
Then the map $W\mapsto \hat{q}:=\argmin_{q\in\PPC} \mJ_W(q)$ is continuous.
\end{theorem}

\begin{proof}
The theorem follows by noting that $W\mapsto \hat{q}:=\argmin_{q\in\PPC} \mJ_W(q)$ can be seen as a composition of two continuous maps, namely the one in Theorem~\ref{thm:W2qcont}  and the one in Theorem~\ref{thm:W2Whom}.
\end{proof}

Next we shall vary also $c$ and $p$, and to this end we introduce a more explicit notation for $\SW$ and $\mathcal{W}$, namely $\ScW=\SW$ in \eqref{eq:SW} and 
\[
\mathcal{W}_{c,p}:=\{W>0\mid \ScW\cap\CP\ne\emptyset, p\not\in\ScW\}.
\]
Then the corresponding set of parameters \eqref{eq:Pcal} for the problem with hard constraints is given by
\begin{equation}
\label{eq:Pcalhard}
\mathcal{P}_{\rm hard}=\{(\PARa)\mid c\in \fC, p\in \PPC\setminus\{0\}, W\in \mathcal{W}_{c,p}\},
\end{equation}
the interior of which is
\begin{equation*}
%\label{eq:Pcalhard_int}
{\rm int}(\mathcal{P}_{\rm hard})=\{(\PARa)\mid c\in \fC, p\in \PP, W\in \mathcal{W}_{c,p}\}.
\end{equation*}
Theorem~\ref{thm:W2Whom} can now be modified accordingly to yield the following theorem, the proof of which is deferred to the appendix.
\begin{theorem}\label{thm:W2Whom2}
Let the map $(c,p,W_{\rm hard})\mapsto W_{\rm soft}$ be given by \eqref{eq:Whard2Wsoft} and the map $(c,p,W_{\rm soft})\mapsto W_{\rm hard}$ by \eqref{eq:Wsoft2Whard}. Then the map that sends $(c,p,W_{\rm hard})\in {\rm int}(\mathcal{P}_{\rm hard})$ to $(c,p,W_{\rm soft})\in {\rm int}(\mathcal{P})$ is a homeomorphism.%, where the component map $W_{\rm hard} \mapsto W_{\rm soft}$ is given by \eqref{eq:Whard2Wsoft} with inverse \eqref{eq:Wsoft2Whard}.
% between $\mathcal{P}_{\rm hard}$ and $\mathcal{P}$, and the inverse is given by \eqref{eq:Wsoft2Whard}.
\end{theorem}
Note that this theorem is not a strict amplification of Theorem~\ref{thm:W2Whom} as we have given up the possibility for $p$ to be on the boundary $\partial \PP$. The same is true for the following modification of Theorem~\ref{thm:W2qcont_hard}.
\begin{theorem}\label{thm:par2qcont_hard} 
Let $\mJ_{\PARa}(q)$ be as in \eqref{eq:Jhard}.
Then the map $(\PARa)\mapsto \hat{q}:=\argmin_{q\in\PPC} \mJ_{\PARa}(q)$ is continuous on  ${\rm int}(\mathcal{P}_{\rm hard})$.
\end{theorem}
\begin{proof}
The theorem follows immediately by noting that $(\PARa_\text{hard})\mapsto \hat{q}$ can be seen as a composition of two continuous maps, namely $(\PARa_\text{hard})\mapsto(\PARa_\text{soft})$ of Theorem~\ref{thm:W2Whom2} and $(\PARa_\text{soft})\mapsto\hat{q}$ of Theorem~\ref{thm:W2qcont}.
%
%the map of Theorem~\ref{thm:W2Whom2}
%
%
%the one in Theorem~\ref{thm:W2qcont}  and the one in Theorem~\ref{thm:W2Whom2}. Note that $\mJ_\PARa(q)$ is defined by \eqref{eq:Jhard}.
\end{proof}
Theorem~\ref{thm:par2qcont_hard}  is a counterpart of Theorem~\ref{thm:W2qcont} for the problem with hard constraints, except that $p$ is restricted to the interior $\PP$. It should be possible to extend the result to hold for all $p\in\PPC\setminus \{0\}$ via a direct proof along the lines of the proof of Theorem~\ref{thm:W2qcont}.

%\begin{proof}
%The proof of Lemma~\ref{lem:W2Jcont} can be trivially modified to show the following. 
%\begin{lemma}\label{lem:W2Jcont_hard}
%With $\mJ_W$ as in \eqref{eq:Jhard}, the  map $W\mapsto\min_{q\in\PPC} \mJ_W(q)$ is continuous.
%\end{lemma}
%
%Using this lemma and replacing Lemma~\ref{lem:Wbounded_hard} with  Lemma~\ref{lem:Wbounded_soft} we readily see that the same proof as that of Theorem~\ref{thm:W2qcont} goes through with trivial modifications.
%\end{proof}

%%%%%%%%%%%%%%%%%%%%%%%%%%%%
% HERE IS THE MTNS-EXAMPLE %
%%%%%%%%%%%%%%%%%%%%%%%%%%%%

\section{Estimating covariances from data}\label{sec:covest}

For a scalar stationary stochastic process $\{y(t);\,t\in\mZ\}$, it is well-known that the biased covariance estimate 
\[
c_k = \frac{1}{N} \sum_{t = 0}^{N - k -1}  y_t\bar{y}_{t+k} ,
\]
based on an observation record $\{ y_t \}_{t = 0}^{N-1}$, yields a positive definite Toeplitz matrix, which is equivalent to $c \in \CP$ \cite[pp. 13-14]{ahiezer1962some} In fact, these estimates correspond to the ones obtained from the periodogram estimate of the spectrum (see, e.g., \cite[Sec. 2.2]{stoica1997introduction}). On the other hand, the Toeplitz matrix of the unbiased estimate 
\[
c_k = \frac{1}{N-k} \sum_{t = 0}^{N -k-1}  y_t \bar{y}_{t+k}
\]
is in general not positive definite.

The same holds in higher dimensions ($d>1$) where the observation record is $\{y_\tb\}_{\tb\in \mZ^d_N}$ with $$\mZ^d_N=\{(\ell_1,\ldots, \ell_d)\,|\, 0\le \ell_j\le N_j-1, j=1,\ldots, d\}.$$ 
The unbiased estimate is then given by
\begin{equation}\label{eq:unbiased_est}
c_\kb = \frac{1}{\prod_{j=1}^d (N_j-|k_j|)} \sum_{\tb \in \mZ^d_{N}} y_\tb \bar{y}_{\tb+\kb},
\end{equation}
and the biased estimate by
\begin{equation}\label{eq:biased_est}
c_\kb = \frac{1}{\prod_{j=1}^d N_j} \sum_{\tb \in \mZ^d_{N}} y_\tb \bar{y}_{\tb+\kb},
\end{equation}
where we define $y_\tb=0$ for $\tb\notin \mZ^d_{N}$. The sequence of unbiased covariance estimates does not in general belong to $\CP$, but  the biased covariance estimates yields $c\in\CP$ also in the multidimensional setting. In fact, this can be seen by noting that the biased estimate corresponds to the Fourier coefficients of the periodogram \cite[Sec. 6.5.1]{dudgeon1984multidimensional}, i.e.,  if the estimates $c_\kb$ are given by \eqref{eq:biased_est}, then 
\begin{equation}\label{eq:periodogram}
\Phi_{\rm periodogram}(\thetab) := \frac{1}{\prod_{j=1}^d N_j} \Big| \sum_{\tb \in \mZ^d_N} y_\tb \expfunt \Big|^2=\sum_{\kb \in \mZ^d_N-\mZ^d_N} c_{\kb} \expfunkm ,
\end{equation}
where $\mZ^d_N-\mZ^d_N$ denotes the Minkowski set difference. This leads to the following lemma.

\begin{lemma}\label{lem:biased}
Given the observed data $\{ y_{\tb} \}_{\tb \in \mZ^d_N}$, let $\{ c_\kb \}_{\kb \in \Lambda}$ be given by \eqref{eq:biased_est}.  Then $c \in \CP$.
\end{lemma}

\begin{proof}
%We first prove the statement when $\Lambda$ is the rectangular grid $\Lambda = \mZ^d_N$.
%_R = \{ \kb \in \mZ^d \; | \; |k_j| \leq N_j, \, j = 1,\ldots,d \}. 
Given $\{ y_{\tb} \}_{\tb \in \mZ^d_N}$, let $c = \{ c_\kb \}_{\kb \in \mZ^d_N}$, where $c_\kb$ be given by \eqref{eq:biased_est}. 
In view of  \eqref{innerprod} and \eqref{eq:periodogram} we have
\[
\langle c, p \rangle = \int_{\mT^d} \frac{1}{\prod_{j=1}^d N_j} \Big| \sum_{\tb \in \mZ^d_N} y_\tb \expfunt \Big|^2 P(e^{i\thetab}) \dm(\thetab),
\]
which is positive for all $p \in \PPC \setminus \{0\}$. Consequently $c \in \CP$. %The result is now easily extended to any grid $\Lambda$ by noting that the only difference in this case is that the sum is instead take over $\kb \in \Lambda$.
\end{proof}

An advantage of the approximate procedures to the rational covariance extension problem is that they can also be used for cases where the biased estimate is not available, e.g., where the covariance is estimated from snapshots.

\section{Application to spectral estimation}\label{sec:example}

As long as we use the biased estimate \eqref{eq:biased_est}, we may apply exact covariance matching as outlined in Section~\ref{sec:prel}, whereas in general approximate covariance matching will be required for biased covariance estimates. However, as will be seen in the following example, approximate covariance matching may sometimes be better even if $c\in\CP$.

In this application it is easy to determine a bound on the acceptable error in the covariance matching, so we use the procedure with hard constraints. Given data generated from a two-dimensional stochastic system, we test three different procedures, namely (i) using the biased estimate and exact matching, (ii) using the biased estimate and the approximate matching \eqref{eq:primal_relax_new}, and (iii)  using the unbiased estimate and the approximate matching \eqref{eq:primal_relax_new}. The procedures are then evaluated by checking the size of the error  between the matched covariances and the true ones from the dynamical system.

\subsection{An example} 

Let $y_{(t_1, t_2)}$ be the steady-state output of a two-dimensional recursive filter driven by a white noise input $u_{(t_1, t_2)}$. Let the transfer function of the recursive filter be
\[
\frac{b(e^{i\theta_1}, e^{i\theta_2})}{a(e^{i\theta_1}, e^{i\theta_2})} = \frac{ \sum_{\kb \in \Lambda_+} b_{\kb} e^{-i(\kb,\thetab)}}{\sum_{\kb \in \Lambda_+} a_{\kb} e^{-i(\kb,\thetab)}},
\]
where $\Lambda_+=\{(k_1,k_2)\in \mZ^2\mid 0\le k_1\le 2, 0\le k_2\le 2 \}$ and the coefficients are given by
$b_{(k_1,k_2)}=B_{k_1+1, k_2+1} $ and $a_{(k_1,k_2)}=A_{k_1+1, k_2+1}$, where
\begin{align*}
B = \!
{\scriptscriptstyle
\begin{bmatrix}
    \,0.9 &  -0.2\phantom{0} &   0.05 \\
    \,0.2 &  \phantom{-} 0.3\phantom{0} &   0.05 \\
   -0.05 &   -0.05 &   0.1\phantom{0}
\end{bmatrix}}, \;
A = \!
{\scriptscriptstyle
\begin{bmatrix}
    \phantom{-}1\phantom{.0} &   \phantom{-}0.1 &   \phantom{-}0.1 \\
   -0.2 &   \phantom{-}0.2 &  -0.1 \\
    \phantom{-}0.4 &  -0.1 &  -0.2
\end{bmatrix}
}.
\end{align*}
The spectral density $\Phi$ of $y_{(t_1, t_2)}$, which is shown in Fig.~\ref{fig:original} and is similar to the one considered in \cite{ringh2015multidimensional}, is given by
\[
\Phi(e^{i\theta_1}, e^{i\theta_2}) = \frac{P(e^{i\theta_1}, e^{i\theta_2})}{Q(e^{i\theta_1}, e^{i\theta_2})} = \left|\frac{b(e^{i\theta_1}, e^{i\theta_2})}{a(e^{i\theta_1}, e^{i\theta_2})}\right|^2,
\]
and hence the index set $\Lambda$ of the coefficients of the trigonometric polynomials $P$ and $Q$ is given by $\Lambda=\Lambda_+-\Lambda_+=\{(k_1,k_2)\in \mZ^2\,|\; |k_1|\le 2, |k_2|\le 2 \}$. Using this example, we perform two different simulation studies.
\begin{figure}
\centering
\includegraphics[width=0.4\textwidth]{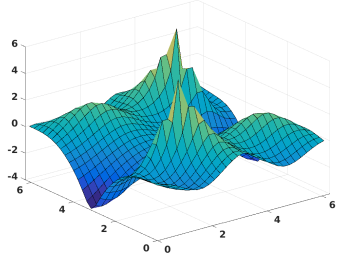}
\caption{Log-plot of the original spectrum.}
\label{fig:original}
\end{figure}

\subsection{First simulation study}
The system was simulated for $500$ time steps along each dimension, starting from $y_{(t_1, t_2)} = u_{(t_1, t_2)} = 0$ whenever either $t_1 < 0$ or $t_2 < 0$. Then covariances were estimated from the $9 \times 9$ last samples, using both the biased and the unbiased estimator. With this covariance data we investigate the  three procedures (i), (ii) and (iii)  described above. In each case, both the maximum entropy (ME) solutions and solutions with the true numerator are computed.\footnote{Maximum entropy: $P\equiv 1$. True numerator: $P=P_{\text{true}}$.} The weighting matrix is taken to be $W = \lambda I$, where $\lambda$ is $\lambda_{\text{biased}} := \| c_{\text{true}} - c_{\text{biased}} \|_2^2$ in procedure (ii) and $\lambda_{\text{unbiased}} := \| c_{\text{true}} - c_{\text{unbiased}} \|_2^2$ in procedure (iii).% 
\footnote{Note that this is the smallest $\lambda$ for which the true covariance sequence belongs to the uncertainty set $\{r\,|\, \|r-c\|_2^2\le \lambda\}$.}
The norm of the error%
\footnote{Here we use the norm of the covariance estimation error as measure of fit. However, note that this is not the only way to compare accuracy of the different methods. The reason for this choice is that comparing the  accuracy of the spectral estimates is not straightforward since it depends on the selected metric or distortion measure.}      
 between the matched covariances and the true ones, $\| \hat{r} - c_{\text{true}} \|_2$, is shown in Table.~\ref{tab:mean_std_2}. The means and standard deviations are computed over the $100$ runs.

The biased covariance estimates belong to the cone $\CP$ (Lemma~\ref{lem:biased}), and therefore procedure (i) can be used. The corresponding error in Table~\ref{tab:mean_std_2} is the statistical error in estimating the covariance. This error is quite large because of a  short data record. Using approximate covariance matching in this case seems to give a worse match.
However, approximate matching of the unbiased covariances gives as good a fit as exact matching of the biased ones. 

\begin{table}[bt]
\caption{Norm differences $\| \hat{r} - c_{\rm {true}} \|_2$ for different solutions in the first simulation setup.}
\label{tab:mean_std_2}
\begin{center}
\begin{tabular}{l l l}
\toprule
                         & Mean     & Std. \\
\midrule
Biased, exact matching   & $3.2374$ & $1.7944$ \\
Biased, approximate matching, ME-solution    & $3.7886$ & $1.3274$  \\ % & $3.8088$ & $1.5932$  \\ These results in the comment here are with P = 1
Biased, approximate matching, using true $P$ & $3.8152$ & $1.6509$  \\
Unbiased, approximate matching, ME-solution    & $3.2575$ & $1.4721$  \\ % & $3.2756$ & $1.7277$  \\ These results in the comment here are with P = 1
Unbiased, approximate matching, using true $P$ & $3.2811$ & $1.7787$  \\
\bottomrule
\end{tabular}
\end{center}
\end{table}

\subsection{Second simulation study}
In this simulation the setup is the same as the previous one, except that the simulation data has been discarded if the {\em unbiased\/} estimate belongs to $\CPC$. To obtain $100$ such data sets, $414$ simulations of the system were needed. (As a comparison, in the previous experiment $23$ out of the $100$ runs resulted in an unbiased estimate outside $\CPC$.)
Again, the norm of the error between matched covariances and the true ones is shown in Table~\ref{tab:mean_std}, and the means and standard deviations are computed over the $100$ runs. 

As before, the biased covariance estimates belong to the cone $\CP$, and therefore procedure (i) can be used.
Comparing this with the results from procedure (ii)  suggests that there may be an advantage not to enforce exact matching, although we know that the data belongs to the cone.
%
% we allow the covariance to move further into the cone, which probably is advantageous from a numerical point of view but clearly will produce a larger error, since the displacement is added to the statistical error.  
%
Regarding procedure (iii), we know that the unbiased covariance estimates do not belong to the cone $\CPC$, hence we need to use approximate covariance matching. In this example, this procedure turns out to give the smallest estimation error.

%\begin{table}[bt]
%\caption{Norm differences $\| r - c_{\rm {true}} \|_2$ for the different solutions.}%, all which have unbiased estimate outside of $\CP$. To generate 100 unbiased estimates outside of $\CP$, 418 simulations were run.}
%\label{tab:mean_std}
%\begin{center}
%\begin{tabular}{l l l}
%\toprule
%                                     & Mean     & Std. \\
%\midrule
%Biased, exact matching   & $2.9245$ & $2.2528$ \\
%Biased, approximate matching, ME-solution    & $4.2295$ & $3.2693$  \\
%Biased, approximate matching, using true $P$ & $4.2131$ & $3.2642$  \\
%Unbiased, approximate matching, ME-solution    & $6.0603$ & $4.9604$  \\
%Unbiased, approximate matching, using true $P$ & $6.1041$ & $4.9751$  \\
%\bottomrule
%\end{tabular}
%\end{center}
%\end{table}

\begin{table}[bt]
\caption{
Norm differences $\| \hat{r} - c_{\rm {true}} \|_2$ for different solutions in the second setup, where all unbiased estimate are outside $\CPC$.}
\label{tab:mean_std}
\begin{center}
\begin{tabular}{l l l}
\toprule
                                     & Mean     & Std. \\
\midrule
Biased, exact matching   & $2.9245$ & $2.2528$ \\
Biased, approximate matching, ME-solution    & $1.9087$ & $1.1324$ \\ % These results are with P = 1. The following values: & $2.2082$ & $0.9298$  are with P = c_0
Biased, approximate matching, using true $P$ & $1.8532$ & $1.1904$  \\
Unbiased, approximate matching, ME-solution    & $1.5018$ & $0.6601$ \\ % These results are with P = 1. The following values: & $1.8509$ & $0.5255$ are with P = c_0
Unbiased, approximate matching, using true $P$ & $1.4451$ & $0.7296$  \\
\bottomrule
\end{tabular}
\end{center}
\end{table}

\section{Application to system identification and texture reconstruction}\label{sec:ex_texture}

Next we apply the theory of this paper to texture generation via Wiener system identification.
Wiener systems form a class of nonlinear dynamical systems consisting of a linear dynamic part composed with a static nonlinearity as illustrated in Figure \ref{fig:blockdiagramRep}. This is a subclass of so called block-oriented systems \cite{billings1980identification}, and Wiener system identification is a well-researched area (see, e.g., \cite{greblicki1992nonparametric} and references therein) that is still very active \cite{lindsten2013bayesian, wahlberg2015identification, abdalmoaty2016simulated}. Here, we use  Wiener systems to model and generate textures.

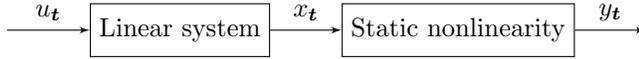
\begin{figure}[t]%[htb]
\centering
\input{figures/blockdiagram.tex}
\caption{A Wiener system with thresholding as static nonlinearity.}
\label{fig:blockdiagramRep}
\end{figure} 

Using dynamical systems for modeling of images and textures is not new and has been considered in, e.g., \cite{Chiuso-F-P-05, picci2008modelling}. The setup presented here is motivated by \cite{barman2015gaussian}, where thresholded Gaussian random fields are used to model porous materials for design of surface structures in pharmaceutical film coatings. Hence we let the static nonlinearity, call it $f$, be a thresholding with unknown thresholding parameter $\tau$. In our previous work \cite{ringh2017further} we applied exact covariance matching to such a problem. However, in general there is no guarantee that the estimated covariance sequence $c$ belongs to the cone $\CP$. Consequently, here we shall use approximate covariance matching instead.

The Wiener system identification can be separated into two parts. We start by identifying  the nonlinear part. Using the notations of Figure~\ref{fig:blockdiagramRep}, let $\{u_\tb; \, \tb \in \mZ^d\}$ be a zero-mean Gaussian white noise input, and let $\{x_\tb; \, \tb \in \mZ^d\}$ be the stationary output of the linear system, which we assume to be normalized so that $c_{\mathbf 0}:= \ExpOp [x_\tb^2]=1$. Moreover, let $y_\tb=f(x_\tb)$ where $f$ is the static nonlinearity
\begin{equation}\label{eq:staticnonlin}
f(x)=\begin{cases}1 & x>\tau \\0 & \mbox{otherwise}\end{cases}
\end{equation}
with unknown thresholding parameter $\tau$. Since $ \ExpOp [y_\tb] = 1-\phi(\tau)$, where $\phi(\tau)$ is the Gaussian cumulative distribution function, an estimate of $\tau$ is given by $\tau_{\rm est}= \phi^{-1}(1- \ExpOp [y_\tb])$.

Now, let $c_\kb^x :=  \ExpOp [x_{\tb+\kb} x_\tb]$ be the covariances of $x_\tb$, and let $c_\kb^y:=  \ExpOp [y_{\tb+\kb} y_\tb] -  \ExpOp [y_{\tb + \kb}] \ExpOp [y_\tb]$ be the covariances of $y_\tb$. As was explained in \cite{ringh2017further}, by using results from \cite{price1958useful} one can obtain a relation between $c_\kb^y$ and $c_\kb^x$, given by 
\begin{equation}
\label{eq:randcrelation}
\begin{split}
c_\kb^y = \int_0^{c_{\kb}^x} \frac{1}{2\pi \sqrt{1-s^2}}\exp\left(-\frac{\tau^2}{1+s}\right)ds .
\end{split}
\end{equation}
This is an invertible map, which we compute numerically, and given $\tau_{\rm est}$ we can thus get estimates of the covariances $c_\kb^x$ from estimates of the covariances $c_\kb^y$.   However, even if $c^y$ is is a biased estimate so that $c^y \in \CP$, $c^x$ may not be a {\em bona fide} covariance sequence.

%As was explained in \cite{ringh2017further}, by using results from \cite{price1958useful} this leads to a relation between the covariances of $x_\tb$, called $c_\kb:=E[x_{\tb+\kb} x_\tb]$, and the covariances of $y_\tb$, called $r_\kb:= E[y_{\tb+\kb} y_\tb] - E[y_{\tb + \kb}]E[y_\tb]$, given by
%\begin{equation}
%\label{eq:randcrelation}
%\begin{split}
%r_\kb = \int_0^{c_{\kb}} \frac{1}{2\pi \sqrt{1-s^2}}\exp\left(-\frac{\tau^2}{1+s}\right)ds .
%\end{split}
%\end{equation}
%This is an invertible map, and given $\tau$ we can thus get estimates of the covariances $c_\kb$ from estimates of the covariances $r_\kb$. However, note that although $r$ might be a bona fide covariance sequence, this is not true for $c$.

\subsection{Identifying the linear system}\label{subsec:linear_syst}

Solving \eqref{eq:primal_relax} or \eqref{eq:primal_relax_new} for a given sequence of covariance estimates $c$, we obtain an estimate of the absolutely continuous part of the power spectrum $\Phi$ of that process. In the case $d=1$, $\Phi =P/Q$ can be factorized as 
\[
\Phi(e^{i\theta}) = \frac{P(e^{i\theta})}{Q(e^{i\theta})} = \frac{|b(e^{i\theta})|^2}{|a(e^{i\theta})|^2},
\]
which provides a transfer function of a corresponding linear system, which fed by a  white noise input  will produce an autoregressive-moving-average (ARMA) process with an output signal with precisely the  power distribution  $\Phi$ in steady state. For $d\geq 2$, a spectral factorization of this kind is not possible in general  \cite{dumitrescu2007positive}, but instead there is always a factorization as a sum-of-several-squares \cite{dritschel2004factorization, geronimo2006factorization},
\[
\Phi(e^{i\thetab}) = \frac{P(e^{i\thetab})}{Q(e^{i\thetab})} = \frac{\sum_{k = 1}^{\ell} |b_k(e^{i\thetab})|^2}{\sum_{k = 1}^m |a_k(e^{i\thetab})|^2},
\]
the interpretation of which in terms of a dynamical system is unclear when $m > 1$. Therefore we resort to a heuristic and apply the factorization procedure in \cite[Theorem 1.1.1]{geronimo2004positive} although some of the conditions 
required to ensure the existence of a spectral factor may not be met.
(See \cite[Section 7]{ringh2015multidimensional} for a more detailed discussion.)

\subsection{Simulation results}
The method, which is summarized in Algorithm~\ref{alg:WienerEst}, is tested on some textures from the Outex database  \cite{ojala2002outex} (available online at \url{http://www.outex.oulu.fi/}). These textures are color images and have thus been converted to binary textures by first converting them to black-and-white and then thresholding them.%
\footnote{The algorithm has been implemented and tested in Matlab, version R2015b. The textures have been normalized to account for light inhomogenities using a reference image available in the database. The conversion from color images to black-and-white images was done with the built-in function \texttt{rgb2gray}, and the threshold level was set to the mean value of the maximum and minimum pixel value in the black-and-white image.}
Three such textures are shown in Figure~\ref{subfig:texture_one} through \ref{subfig:texture_three}.

\renewcommand{\algorithmicrequire}{\textbf{Input:}}
\renewcommand{\algorithmicensure}{\textbf{Output:}}
\algsetup{indent=12pt}

\begin{algorithm}
\caption{}%\ar{Approach for Wiener system identification?}}
\label{alg:WienerEst}
\begin{algorithmic}[1]
\REQUIRE $(y_\tb)$
\STATE Estimate threshold parameter: $\tau_{\rm est}= \phi^{-1}(1-E[y_\tb])$
\STATE Estimate covariances: $c_\kb^y := E[y_{\tb+\kb} y_\tb] - E[y_{\tb + \kb}]E[y_\tb]$
\STATE Compute covariances $c_\kb^x := E[x_{\tb+\kb} x_\tb]$  by using \eqref{eq:randcrelation}
\STATE Estimate a rational spectrum using Theorem~\ref{thm:softcontraints} or \ref{thm:hardcontraints}
\STATE Apply the factorization procedure in \cite[Theorem 1.1.1]{geronimo2004positive}
%\WHILE{$||\nabla \mathbb{J}|| > \epsilon$}
%\STATE Compute $\tilde{c}^j$, $\tilde{m}^j$ and $\varepsilon^j$
%\IF{$f_{i} > f_{i-1}$  \OR Solution infeasible}
%\STATE Break while loop
%\ENDIF
%\ENDWHILE
%\RETURN
\ENSURE $\tau_{\rm est}$, coefficients for the linear dynamical system
\end{algorithmic}
\end{algorithm}
In this example there is no natural bound on the error, so we use the problem with soft constraints, for which we choose the weight $W = \lambda I$ with $\lambda=0.01$ for all data sets. Moreover, we do maximum-entropy reconstructions, i.e., we set the prior to $P \equiv 1$. The optimization problems are then solved by first discretizing the grid $\mT^2$, in this case in $50 \times 50$ points (cf. \cite[Theorem 2.6]{ringh2015multidimensional}), and solving the corresponding problems using the CVX toolbox \cite{cvx, grant2008graph}. The reconstructions are shown in Figures~\ref{subfig:recon_texture_one}~-~\ref{subfig:recon_texture_three}. Each reconstruction seems to provide a reasonable visual representation of the structure of the corresponding original. This is especially the case for the second texture.

\begin{figure*}%
  \centering
  \subfloat[First texture.\label{subfig:texture_one}]{\includegraphics[width=0.3\textwidth]{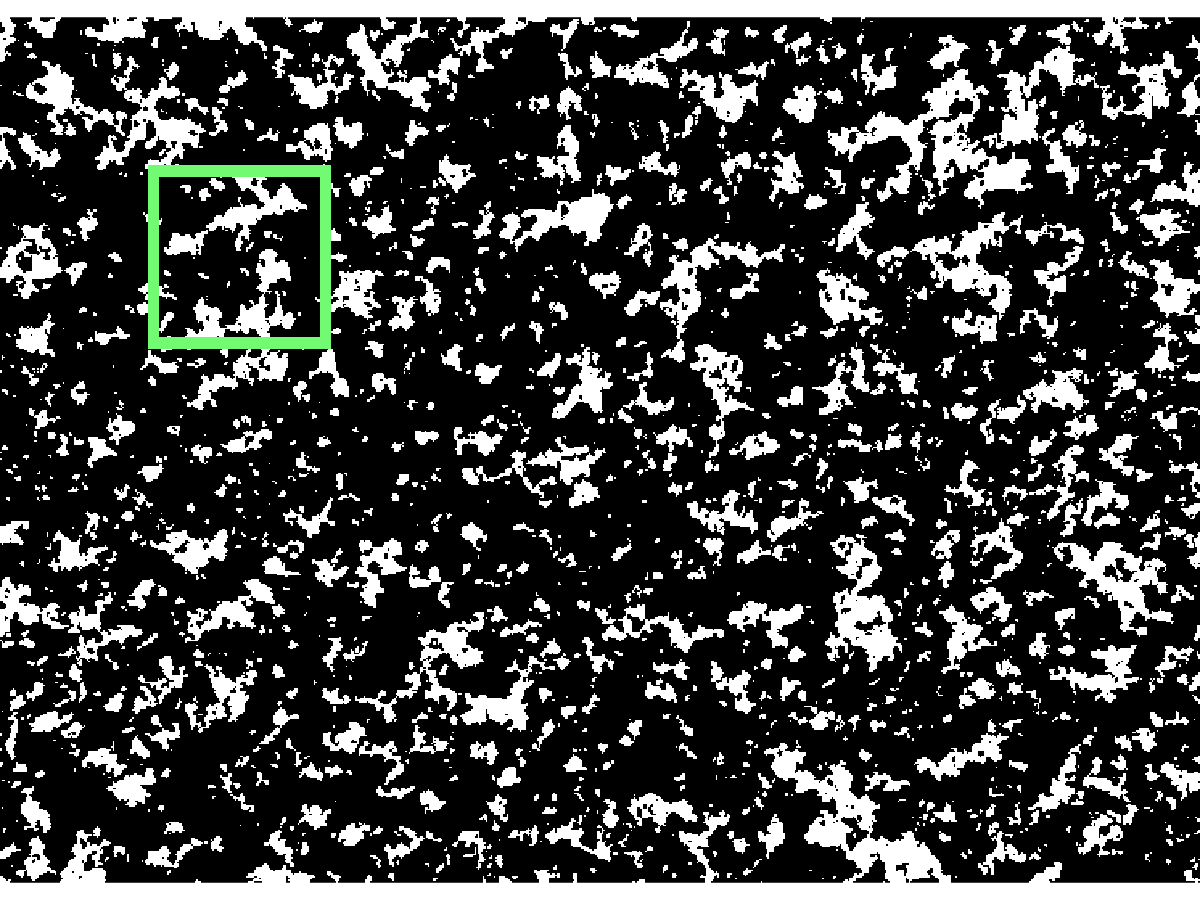}}%
  \hfil
  \subfloat[Second texture.\label{subfig:texture_two}]{\includegraphics[width=0.3\textwidth]{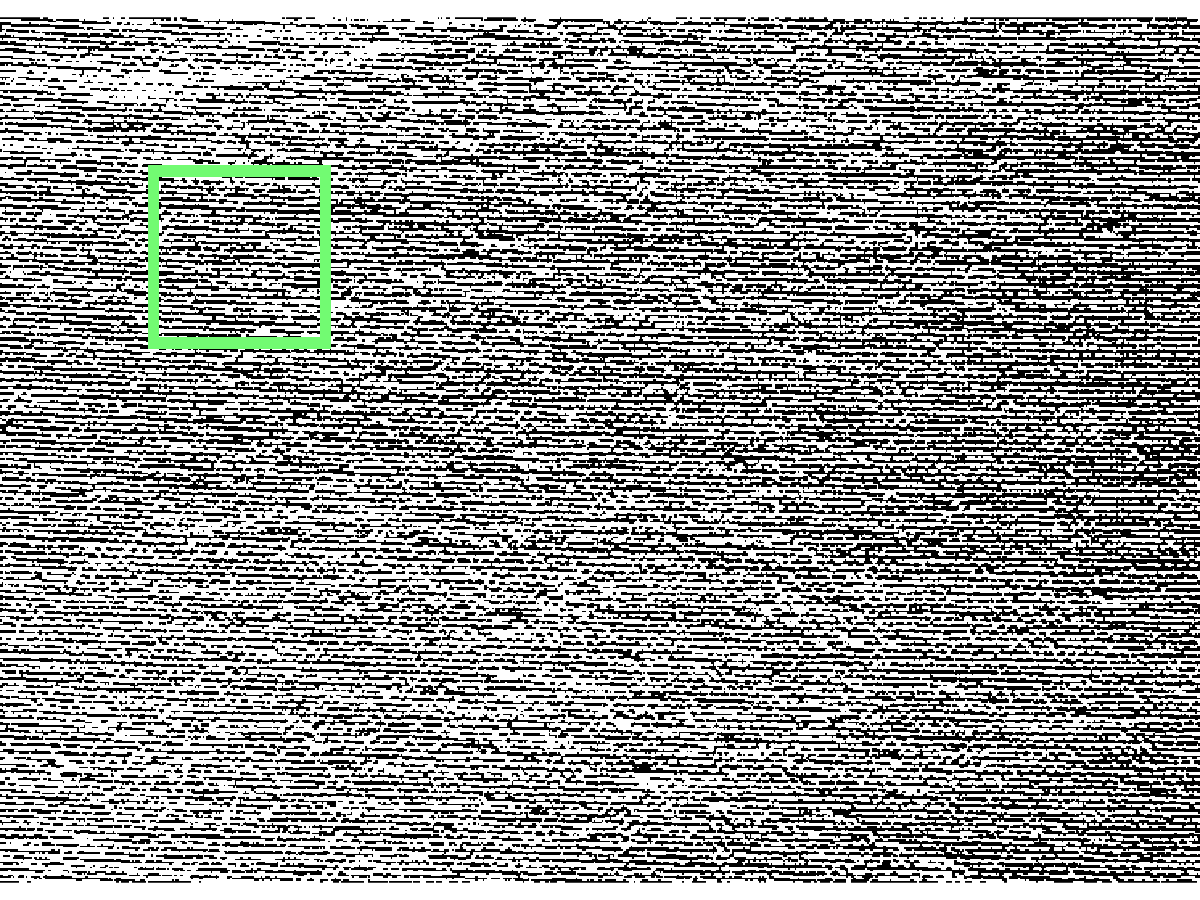}}%
  \hfil
  \subfloat[Third texture.\label{subfig:texture_three}]{\includegraphics[width=0.3\textwidth]{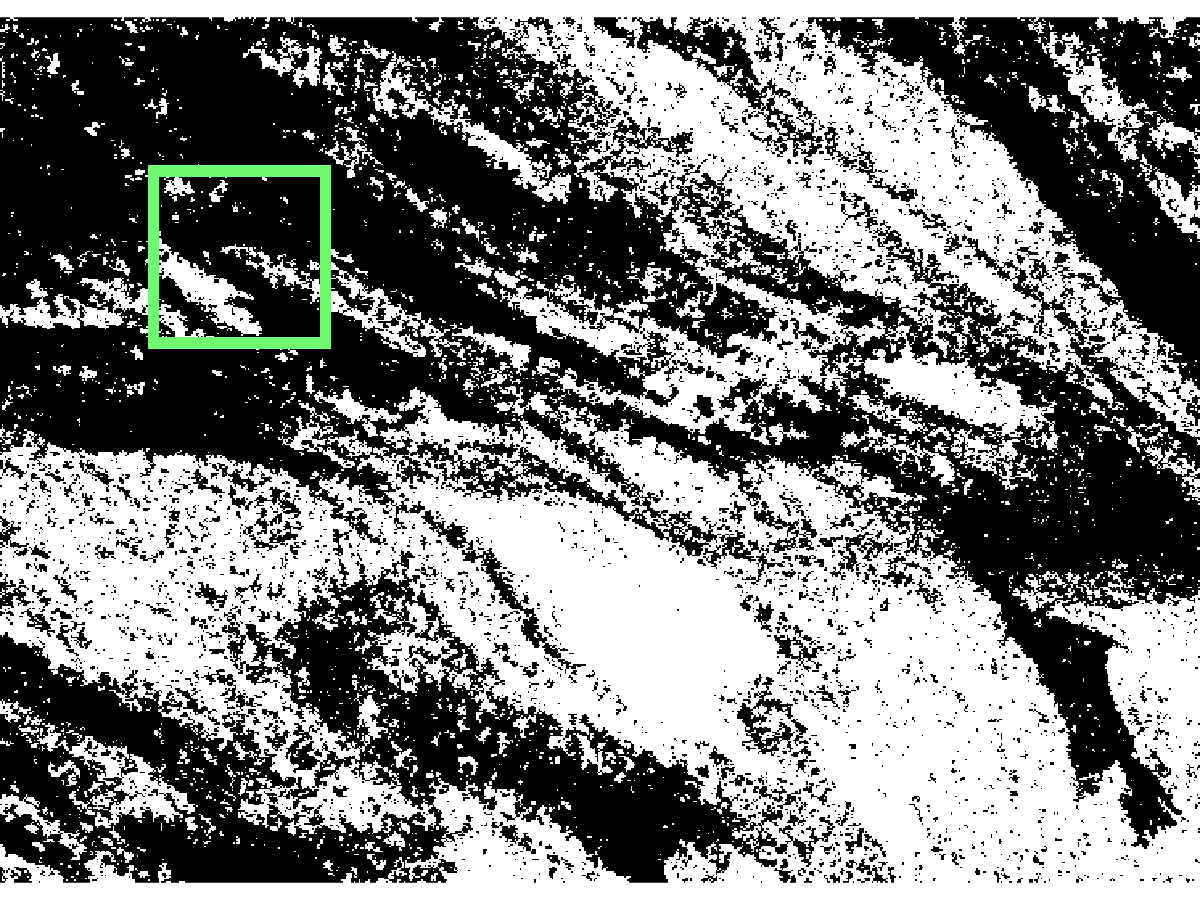}}%
  \hfil 
  \subfloat[Reconstruction of \ref{subfig:texture_one}.\label{subfig:recon_texture_one}]{\includegraphics[width=0.3\textwidth]{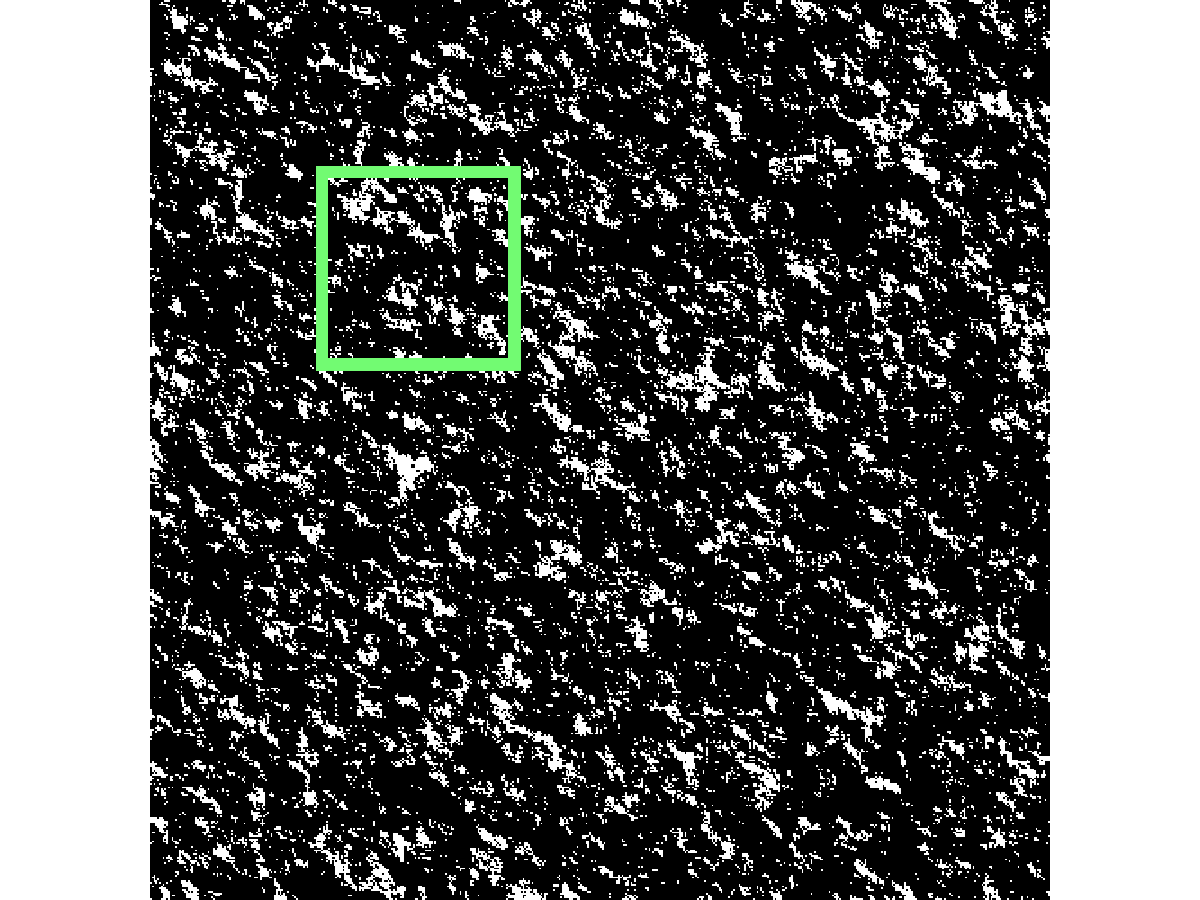}}%
  \hfil 
  \subfloat[Reconstruction of \ref{subfig:texture_two}.\label{subfig:recon_texture_two}]{\includegraphics[width=0.3\textwidth]{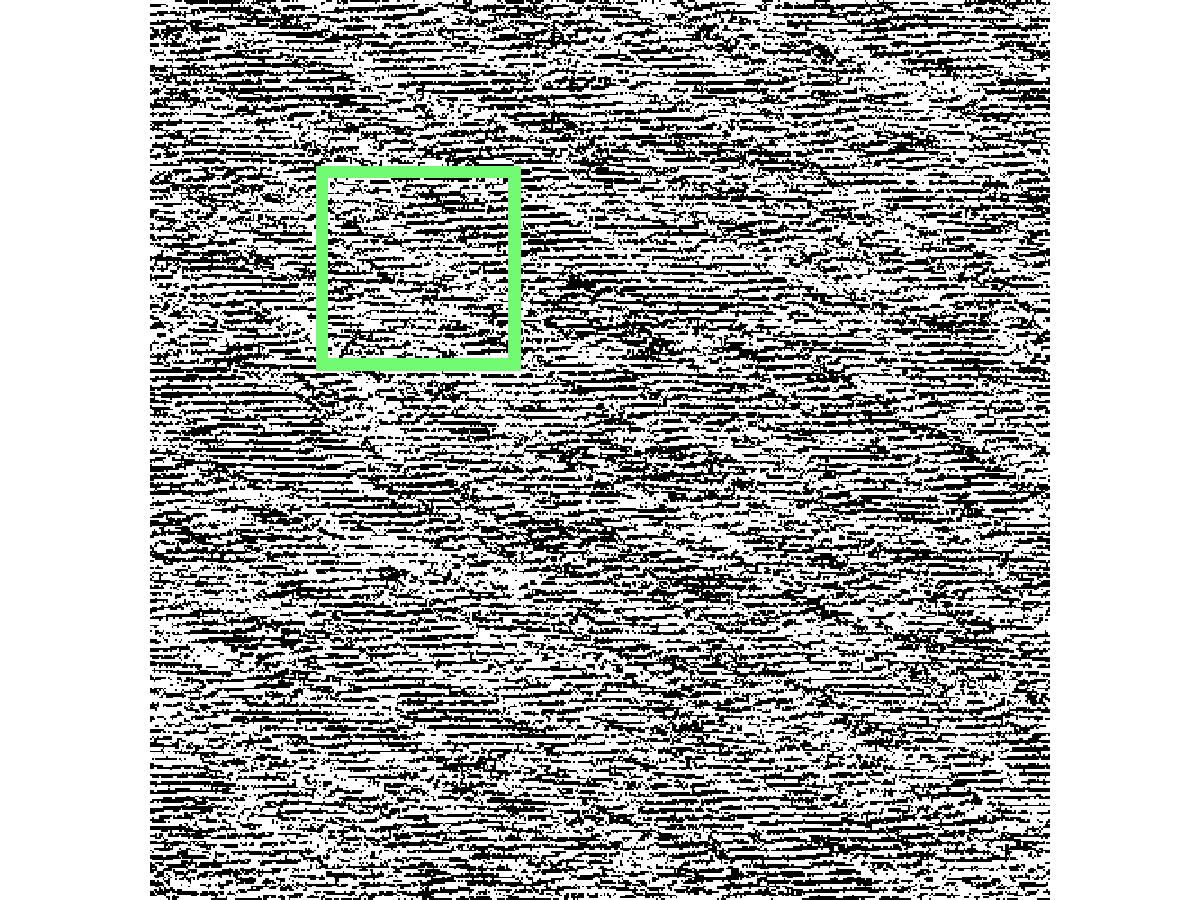}}%
  \hfil
  \subfloat[Reconstruction of \ref{subfig:texture_three}.\label{subfig:recon_texture_three}]{\includegraphics[width=0.3\textwidth]{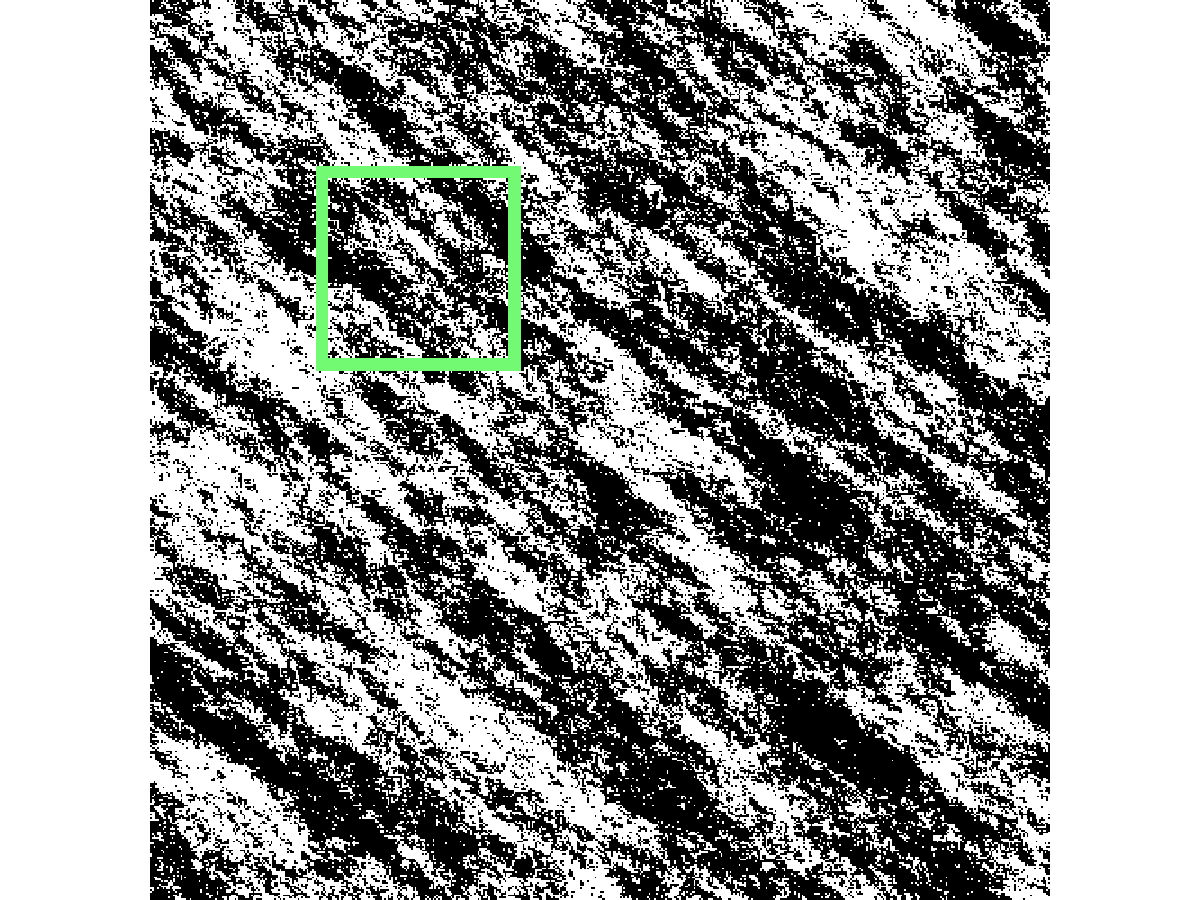}}%
  \hfil
  \subfloat[Close-up of \ref{subfig:texture_one}.\label{subfig:texture_one_closeup}]{\includegraphics[width=0.3\textwidth]{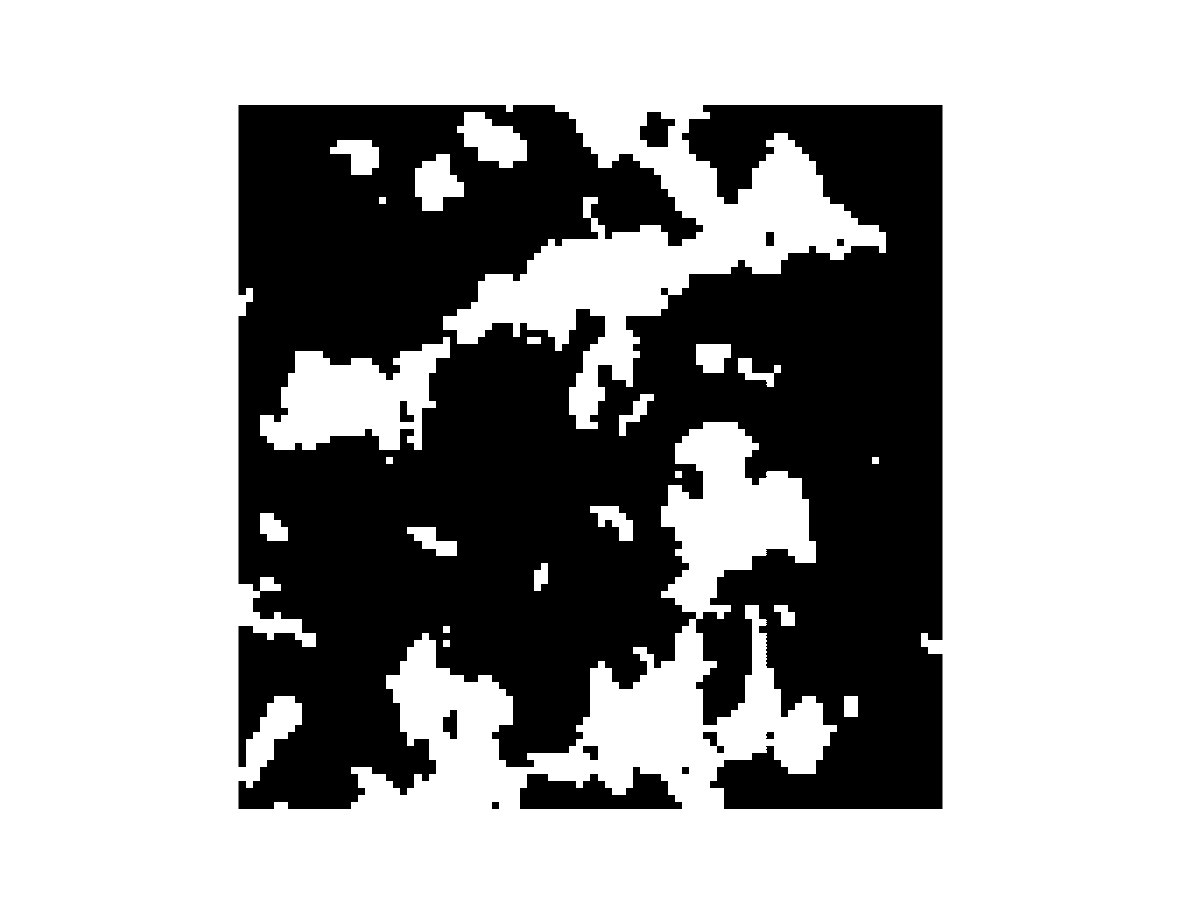}}%
  \hfil
  \subfloat[Close-up of \ref{subfig:texture_two}.\label{subfig:texture_two_closeup}]{\includegraphics[width=0.3\textwidth]{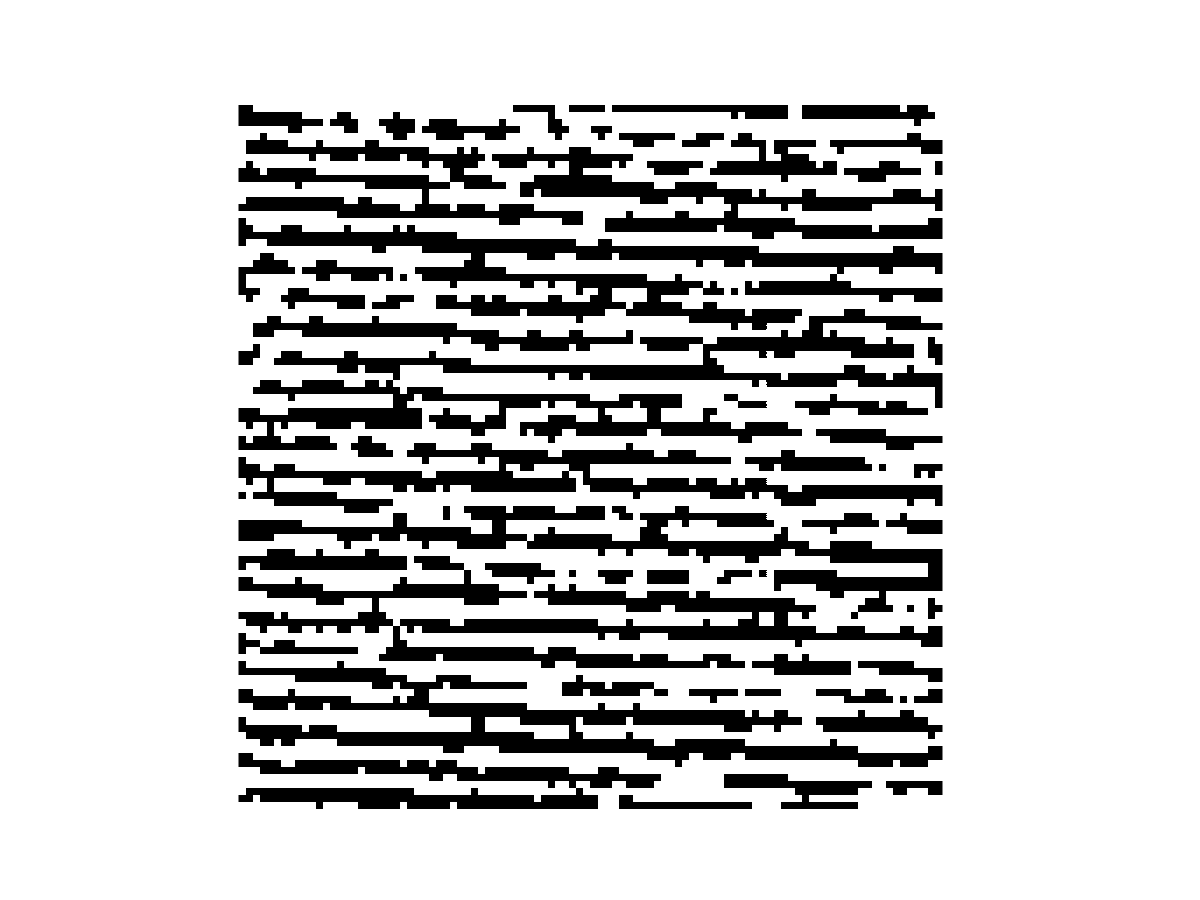}}%
  \hfil
  \subfloat[Close-up of \ref{subfig:texture_three}.\label{subfig:texture_three_closeup}]{\includegraphics[width=0.3\textwidth]{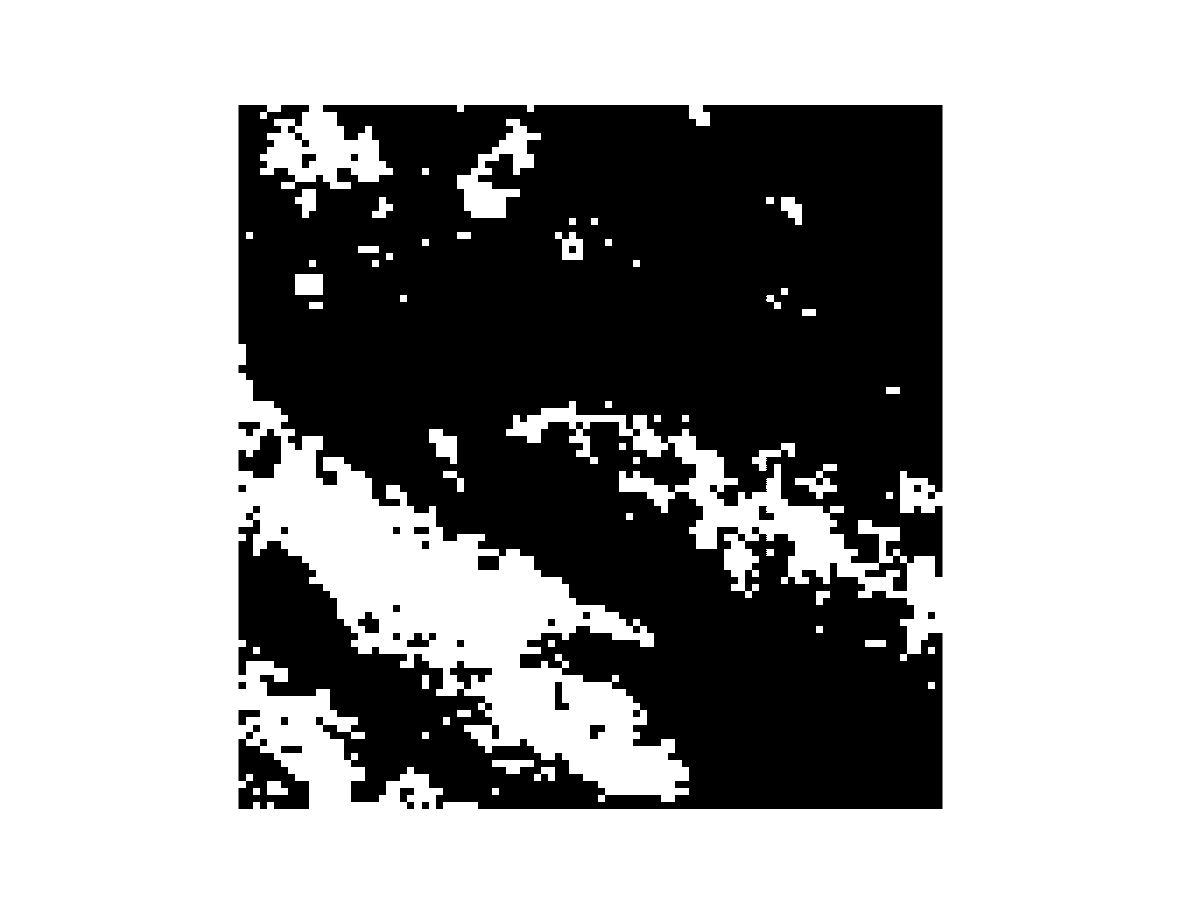}}%
  \hfil 
  \subfloat[Close-up of \ref{subfig:recon_texture_one}.\label{subfig:recon_texture_one_closeup}]{\includegraphics[width=0.3\textwidth]{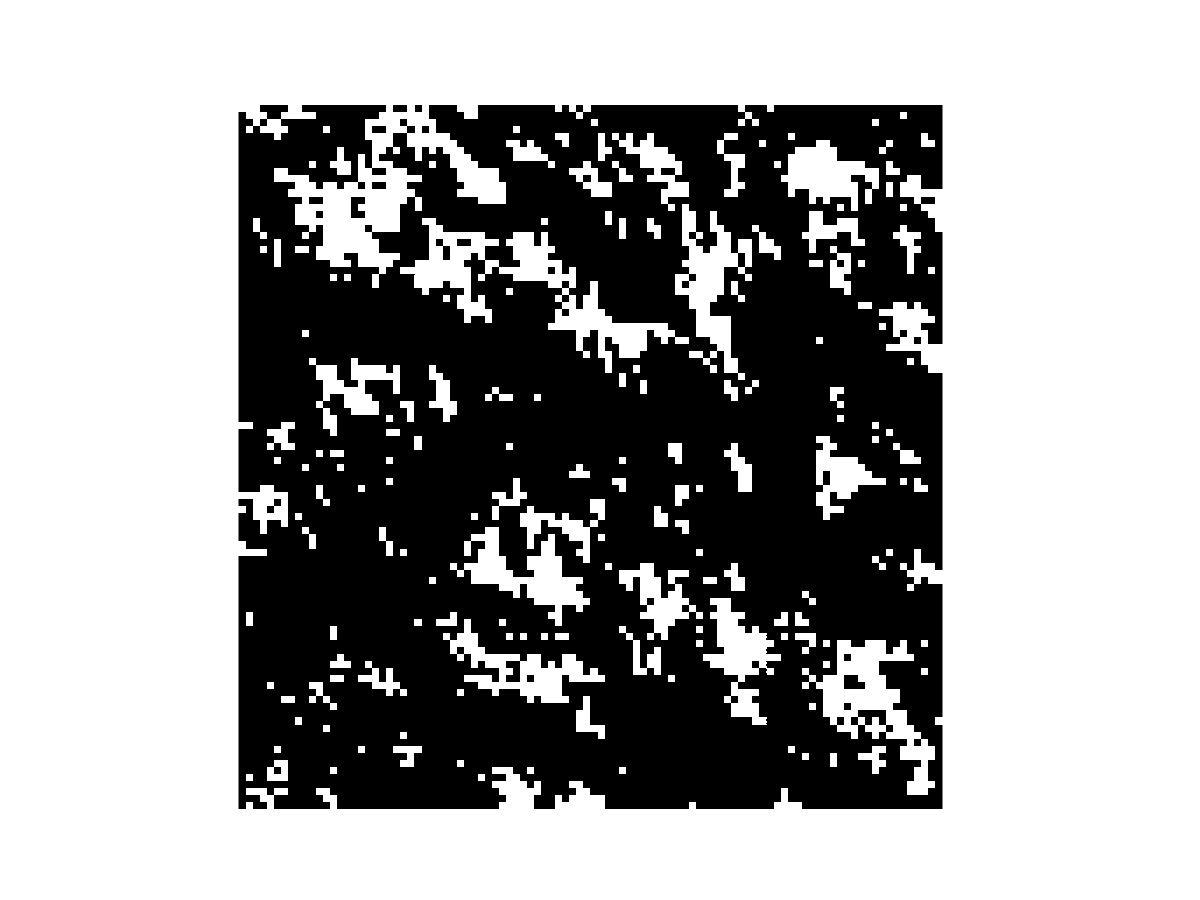}}%
  \hfil 
  \subfloat[Close-up of \ref{subfig:recon_texture_two}.\label{subfig:recon_texture_two_closeup}]{\includegraphics[width=0.3\textwidth]{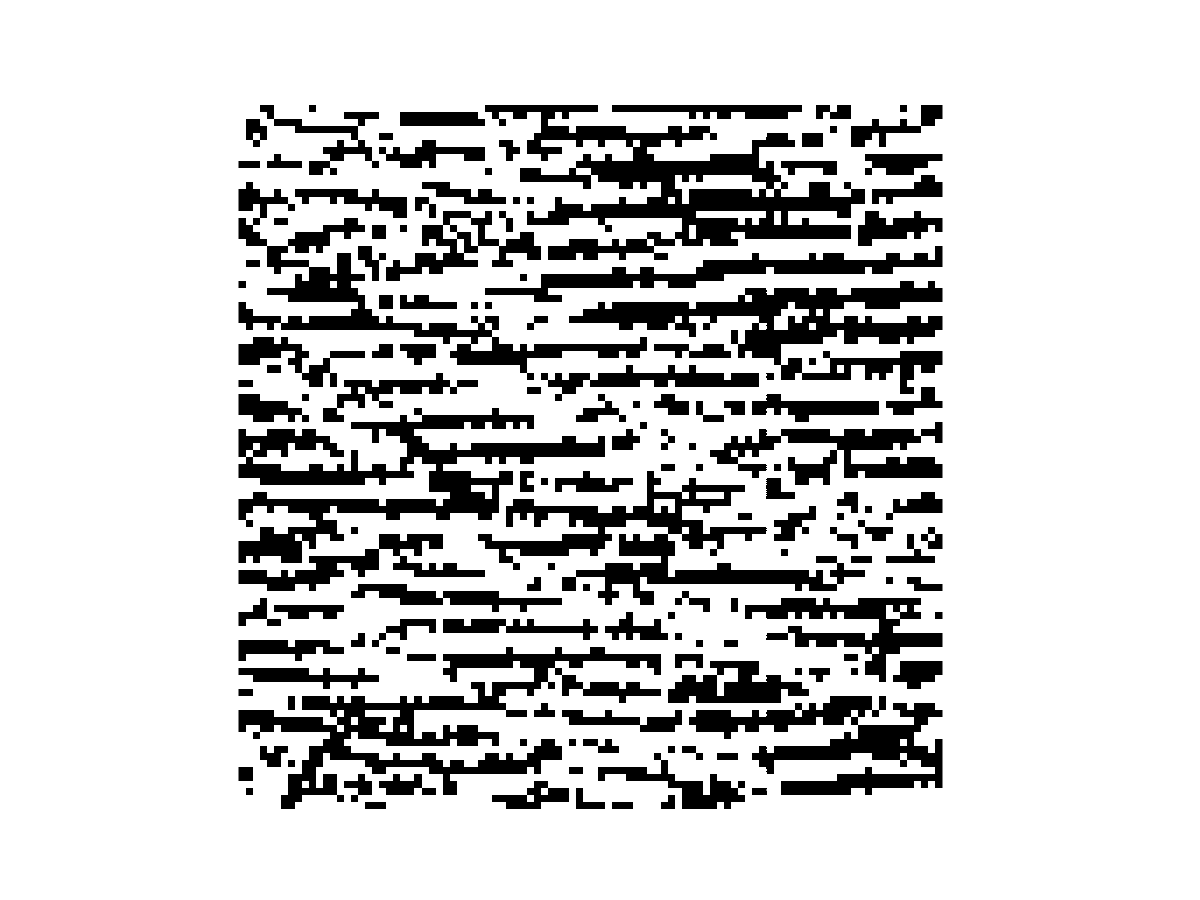}}%
  \hfil
  \subfloat[Close-up of \ref{subfig:recon_texture_three}.\label{subfig:recon_texture_three_closeup}]{\includegraphics[width=0.3\textwidth]{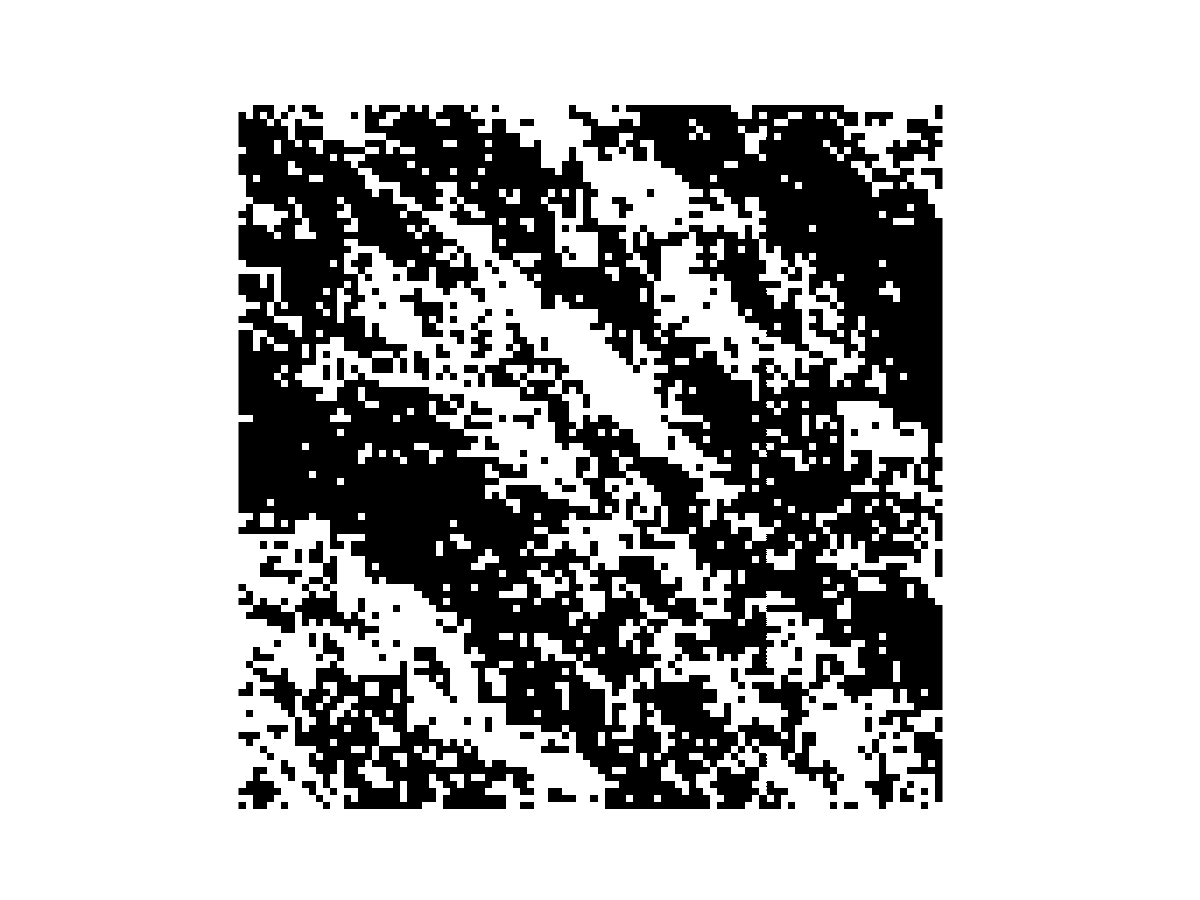}}%
    \caption{In Figures~\ref{subfig:texture_one}~-~\ref{subfig:texture_three} three different binary textures, of size $1200 \times 900$ pixels, are shown. These are obtained from the textures granite001-inca-100dpi-00, paper010-inca-100dpi-00, and plastic008-inca-100dpi-00 in the Outex database, respectively. The textures in Figures~\ref{subfig:texture_one}~-~\ref{subfig:texture_three} are used as input $(y_\tb)$ to Algorithm~\ref{alg:WienerEst} and in Figures~\ref{subfig:recon_texture_one}~-~\ref{subfig:recon_texture_three} the corresponding reconstructed textures of size $500\times 500$ are shown. In Figures~\ref{subfig:texture_one_closeup}~-~\ref{subfig:recon_texture_three_closeup}  close-ups of size $100\times 100$ are shown of the original and reconstructed textures (areas marked in Figures~\ref{subfig:texture_one}~-~\ref{subfig:recon_texture_three}).
  }%
  \label{fig:texture_generation}%
\end{figure*}

\section{Conclusions}
In this work we extend the results of our previous paper \cite{ringh2015multidimensional}  on the multidimensional rational covariance extension problem to allow for approximate covariance matching. We have provided two formulations to this problem, and we have shown that they are connected via a homeomorphism. In both formulations we have used  weighted 2-norms to quantify the missmatch of the estimated covariances. However,  we expect that by suitable modifications of the proofs similar results can be derived for other norms, since all norms have directional derivatives in each point \cite[p. 49]{deimling1985nonlinear}. 

These results  provide a procedure for multidimensional spectral estimation, but in order to obtain a complete theory for multidimensional system identification and realization theory  there are still some open problems, such as spectral factorization and interpretations in terms of multidimensional stochastic systems, as briefly discussed in Section~\ref{subsec:linear_syst}.

\appendix

\section{Deferred proofs}

Let $B_\rho(x^{(0)})$ denote the closed ball $\{ x\in X\mid \|x-x^{(0)}\|_X\leq\rho\}$, where $X$ is either a set of vectors or a set of matrices depending on then context. The  norm $\|\cdot\|_X$ is the Euclidean norm for vectors and Frobenius norm for matrices.

%\subsection*{Lemmas for Theorem~\ref{thm:W2qcont}}
%
%\mbox{}
%
\begin{lemma}\label{lem:W2Jcont}
Let $\mathcal{P}$ be given  by \eqref{eq:Pcal} and $\mJ_{\PARa}$ by \eqref{eq:Jsoft}. Furthermore, let $\hat{q}:=\min_{q\in\PPC} \mJ_{\PARa}(q)$. Then  the map $(\PARa)\mapsto \mJ_{\PARa}(\hat{q})$ is continuous for  $(\PARa)\in\mathcal{P}$. Moreover, for any compact $K\subset\mathcal{P}$, the corresponding set of optimal solutions $\hat{q}$ is bounded.
\end{lemma}

\begin{proof}
The proof follows along the lines of Lemma 7.2 and  Proposition 7.4 in \cite{karlsson2015themultidimensional}. Let $(\PARb)\in\mathcal{P}$ be arbitrary and let 
\begin{align*}
&\tilde B_\rho(\PARb):=B_\rho(\cb) \times \left(B_\rho(\pb)\cap\PPC\right) \times B_\rho(\Wb),
%\\
%&\phantom{xx}\{(\PARa)\mid \|c-\cb\|_2\le\rho,\, \|p-\pb\|_2\le\rho,\, \|W-\Wb \|_F \le\rho, \, p\in \PPC\}
\end{align*}
where $\rho>0$ is chosen so that $\tilde B_\rho(\PARb)\subset \mathcal{P}$, i.e., $\rho<\|\pb\|_2$ and $W>0$ for all $\|W-\Wb\|_F\leq \rho$.  First we will show that the minimizer $\hat q_{\PARa}$ of $\mathbb{J}_{\PARa}$ is bounded for all  $(\PARa)\in \tilde B_\rho(\PARb)$. To this end, note that by optimality
\[
\mathbb{J}_{\PARa}(\hat q_{\PARa})\le \mathbb{J}_{\PARa}(e)=\langle c, e \rangle - \int_{\mT^d} P \log 1\, \dm+ \frac{1}{2}\|e - e\|_{W}^2= c_\nollb ,
\]
and hence $\mathbb{J}_{c, p, W}(p)$ is bounded from above on the compact set $\tilde B_\rho(\PARb)$. Consequently, by using the same inequality as in the proof of \cite[Lemma~7.1]{karlsson2015themultidimensional}, we see that 
\[
c_\nollb\geq \mathbb{J}_{\PARa}(\hat q_{\PARa})\ge \langle c, \hat q_{\PARa} \rangle -  \|P\|_1 \log \|\hat Q_{\PARa}\|_\infty\, \dm+ \frac{1}{2}\|\hat q_{\PARa} - e\|_{W}^2.
\]
Due to norm equivalence between $\|Q\|_\infty$ and $\|q\|_W$, and since the quadratic term is dominating, the norm of $\hat q_{\PARa}$ is bounded in the set $\tilde B_\rho(\PARb)$.

Now, let $K\subset\mathcal{P}$ be compact. We want to show that $\hat q$ is bounded on $K$. Assume it is not. Then let $(\PARk)\in K$ be a sequence with $\|\hat q_k\|\to \infty$. Since $K$ is compact there is a converging subsequence $(\PARk)\to (\PARa)\in K$ with $\|\hat q_k\|\to \infty$. Since $(\PARa)\in K$ there is a $\rho>0$ such that $\tilde B_\rho(\PARa)\subset \mathcal{P}$. However, all but finitely many points $(\PARk)$ belong to $\tilde B_\rho(\PARa)$, and 
since $\hat q_k$ is bounded for all $(\PARk)\in \tilde B_\rho(\PARa)$, we cannot have $\|\hat q_k\|\to \infty$.

Next, let $(\PARc), (\PARd)\in \tilde B_\rho(\PARb)$ and let $\hat{q}_{1}, \hat{q}_2\in\bar{\mathfrak{P}}_+$ be the unique minimizers of $\mathbb{J}_{\PARc}$ and $\mathbb{J}_{\PARd}$, respectively. 
Choose a $q_0 \in \mathfrak{P}_+$ and note that $Q_0$ is strictly positive and bounded. By  optimality, 
\begin{subequations}\label{optimality}
\begin{eqnarray}
&\mathbb{J}_{\PARc}(\hat{q}_1)\leq \mathbb{J}_{\PARc}(\hat{q}_2+\varepsilon q_0) & \label{optimality1}\\
&\mathbb{J}_{\PARd}(\hat{q}_2)\leq \mathbb{J}_{\PARd}(\hat{q}_1+\varepsilon q_0) &\label{optimality2}
\end{eqnarray}
\end{subequations}
 for all $\varepsilon >0$. Hence, if  we can show that, for any $\delta >0$, there is an $\varepsilon>0$ and a $\tilde{\rho}>0$ such that 
\begin{subequations}\label{variation}
\begin{eqnarray}
&|\mathbb{J}_{\PARd}(\hat{q}_1+\varepsilon q_0) - \mathbb{J}_{\PARc}(\hat{q}_1)|\leq \delta & \label{variation1}\\
&|\mathbb{J}_{\PARc}(\hat{q}_2+\varepsilon q_0)-\mathbb{J}_{\PARd}(\hat{q}_2)|\leq \delta &\label{variation2}
\end{eqnarray}
\end{subequations}
hold whenever $\|\cc-\cd\|_2\le\tilde{\rho}$, $\|\pc-\pd\|_2\le\tilde{\rho}$ and $\|\Wc-\Wd\|_F\le\tilde{\rho}$, then this would imply that
\begin{displaymath}
\mathbb{J}_{\PARd}(\hat{q}_2)-\delta\leq \mathbb{J}_{\PARc}(\hat{q}_1)\leq \mathbb{J}_{\PARd}(\hat{q}_2)+\delta ,
\end{displaymath}
showing that the optimal value is continuous in $\PARc$. 
The lower bound is obtained by using \eqref{variation1} and \eqref{optimality2}, and the upper bound is obtained from \eqref{optimality1} and \eqref{variation2}. To prove \eqref{variation1}, 
we note that
\begin{align}
%\begin{split}
&|\mathbb{J}_{\PARd}(\hat{q}_1+\varepsilon q_0) - \mathbb{J}_{\PARc}(\hat{q}_1)|\notag\\
&\phantom{xx}=\bigg|\langle \cd-\cc, \hat q_1\rangle+\langle \cd, \varepsilon q_0 \rangle - \int_{\mT^d} \Pc \log\left(1 + \frac{\varepsilon Q_0}{\hat{Q}_1}\right) \dm \notag\\
&\phantom{xxxx}-\int_{\mT^d} (\Pd-\Pc) \log\left(\hat{Q}_1+\varepsilon Q_0\right) \dm +\tfrac12\| \hat{q}_1+\varepsilon q_0 -e\|_{\Wd}^2  -\tfrac12\|\hat{q}_1 -e\|_{\Wc}^2\bigg| \notag\\
&\phantom{xx}\leq \|\cd-\cc\|_2\| \hat q_1\|_2+ \varepsilon \left( \langle \cd, q_0 \rangle + \int_{\mT^d} \Pc \frac{Q_0}{\hat{Q}_1} \dm \right)\notag\\ 
&\phantom{xxxx}+\|\Pd-\Pc\|_2 \|\log(\hat{Q}_1+\varepsilon Q_0)\|_\infty +\tfrac12\left|\| \hat{q}_1+\varepsilon q_0 -e\|_{\Wd}^2- \|\hat{q}_1 -e\|_{\Wd}^2\right|\notag\\
& \phantom{xxxx}  +\tfrac12\left|\|\hat{q}_1 -e\|_{\Wd}^2 -\|\hat{q}_1 -e\|_{\Wc}^2\right|. \label{eq:ugly}
%\end{split}
\end{align}
Next we observe that $$0\le \int_{\mT^d} \Pc \frac{Q_0}{\hat{Q}_1} \dm =\langle \hat r_1-\hat c_1,q_0\rangle\le\langle c_1+\Wc(\hat q_1-e),q_0\rangle$$ by the KKT conditions \eqref{eq:opt_cond_relax} and the fact that $q_0\in\PP$, $\hat c_1\in \CPC$. Hence $\varepsilon$ can be selected small enough for  the second and fourth term in \eqref{eq:ugly} each to be bounded by $\delta/5$ for any $(\PARc), (\PARd)\in B_{\rho}(\PARb)$. Each of the remaining terms can now be bounded by $\delta/5$ by selecting $\tilde\rho$ sufficiently small.
 Hence \eqref{variation1} follows. This also proves \eqref{variation2}. 
\end{proof}

\begin{lemma}\label{lem:qhatbounded}
Let $\mathcal{P}_{\rm hard}$ be given  by \eqref{eq:Pcalhard} and $\mJ_{\PARa}$ by \eqref{eq:moddual}. Furthermore, let $\hat{q}:=\min_{q\in\PPC} \mJ_{\PARa}(q)$. Then for any compact $K\subset\mathcal{P}_{\rm hard}$, the corresponding set of optimal solutions $\hat{q}$ is bounded.
\end{lemma}

\begin{proof}
The proof follows closely the proof of the corresponding part of Lemma~\ref{lem:W2Jcont}. Let $(\PARb)\in\mathcal{P}_{\rm hard}$ be arbitrary and let 
\begin{align*}
&\tilde B_\rho(\PARb):=B_\rho(\cb) \times \left(B_\rho(\pb)\cap\PPC\right) \times B_\rho(\Wb),
%\\
%&\phantom{xx}\{(\PARa)\mid \|c-\cb\|_2\le\rho,\, \|p-\pb\|_2\le\rho,\, \|W-\Wb \|_F \le\rho, \, p\in \PPC\}
\end{align*}
where $\rho>0$ is chosen so that $\tilde B_\rho(\PARb)\subset \mathcal{P}_{\rm hard}$. %, i.e., $\rho<\|\pb\|_2$ and $W>0$ for all $\|W-\Wb\|_F\leq \rho$.  
To see that the minimizer $\hat q_{\PARa}$ of $\mathbb{J}_{\PARa}$ is bounded for all  $(\PARa)\in \tilde B_\rho(\PARb)$, first note that by optimality
\[
\mathbb{J}_{\PARa}(\hat q_{\PARa})\le \mathbb{J}_{\PARa}(e)=\langle c, e \rangle - \int_{\mT^d} P \log 1\, \dm + \|e - e\|_{W}= c_\nollb ,
\]
and hence $\mathbb{J}_{c, p, W}(p)$ is bounded from above on the compact set $\tilde B_\rho(\PARb)$.

Now let $h(q):= \langle c, q-e\rangle +\|q-e\|_W$, as in the proof of Theorem ~\ref{thm:hardcontraints}. Following the same line of argument as in that proof, we see that $h(q)>0$  for all $q\in\PPC$ and $(\PARa)\in \tilde B_\rho(\PARb)$. Since $h$ is continuous in the arguments $(q, \PARa)$, it has a minimum $\varepsilon>0$ on the compact set of tuples $(q, \PARa)$ such that $q\in\PPC\setminus\{0\}, \|q-e\|_\infty =1,$ and $(\PARa)\in \tilde B_\rho(\PARb)$ hold. Thus the second half of inequality \eqref{eq:Jhard_ineq} still holds, i.e., 
\begin{equation}
\mathbb{J}_{\PARa}(q) \geq  \frac{\varepsilon}{|\Lambda |} \|Q\|_\infty -\int_{\mathbb{T}^d}\! \!  P \log \|Q\|_\infty \dm  - \varepsilon\|e\|_\infty \,
\end{equation}
for all $q$. This is true in particular for $\hat q_{\PARa}$, thus 
\[
c_\nollb\geq \mathbb{J}_{\PARa}(\hat q_{\PARa})\ge 
\frac{\varepsilon}{|\Lambda |} \|Q_{\PARa}\|_\infty -\int_{\mathbb{T}^d}\! \!  P \log \|Q_{\PARa}\|_\infty \dm  - \varepsilon\|e\|_\infty. 
\]
Since the linear growth dominates the logarithmic growth, the norm of $\hat q_{\PARa}$ is bounded on the set $\tilde B_\rho(\PARb)$. The proof now follows {\em verbatim} from the argument in the second paragraph in the proof of Lemma~\ref{lem:W2Jcont}.
\end{proof}

\begin{proof}[{Proof of Theorem~\ref{thm:W2Whom2}}]
%\begin{proofWithName}{Proof of Theorem~\ref{thm:W2Whom2}}
This is a modification of the proof of Theorem~\ref{thm:W2Whom}, again utilizing \cite[Lemma 2.3]{byrnes2007interior},  where we replace the map $f$ defined by 
$\mathcal{W}\ni W_{\rm hard}\mapsto W_{\rm soft}\in \{W\mid W>0\}$ and redefine it with the map ${\rm int}(\mathcal{P}_{\rm hard})\ni (c,p,W_{\rm hard})\mapsto (c,p,W_{\rm soft})\in {\rm int}(\mathcal{P})$.
To show that $f$ is a homeomorphism we need to show that the map is proper. 
To this end, we take a compact set $K\subset{\rm int}(\mathcal{P}_{\rm hard})$ and show that $f^{-1}(K)$ is also compact. 
Again, there are two ways this could fail. First, the preimage could contain a singular semidefinite matrix. However this is impossible by \eqref{eq:Wsoft2Whard}, since $\|\hat{q}\|_\infty$ is bounded for $(c,p,W_{\text{hard}})\in K$ (Lemma~\ref{lem:qhatbounded}) and a nonzero scaling of a singular matrix cannot be nonsingular. 
Secondly, $\|W_{\rm soft}\|_F$ could tend to infinity. However, this is also impossible. To see this, we first show that there is a $\kappa >0$ such that $\|p-r\|_{W_{\rm hard}^{-1}}\geq \kappa$ for all $r\in\mathfrak{S}_{c,W_{\rm hard}}$ and all $(c,p,W_{\rm hard})\in K$. 
Again, using the triangle inequality $\|p-r\|_{W_{\rm hard}^{-1}}\geq\|p-c\|_{W_{\rm hard}^{-1}}-\|c-r\|_{W_{\rm hard}^{-1}}$, we observe that the minimum of $\|p-r\|_{W_{\rm hard}^{-1}}$ over all $(c,p,W_{\rm hard})\in K$ and $r$ satisfying the constraint $\|r-c\|_{W_{\rm hard}^{-1}}\leq 1$ is bounded by 
\begin{equation*}
\kappa:= \min_{(c,p,W_{\rm hard})\in K} \|p-c\|_{W_{\rm hard}^{-1}} -1.
\end{equation*}
 The minimum is attained,  as $K$ is compact, and positive, since  $p\not\in\bigcup_{(c,W_{\rm hard})\in K}\ScW$.
The remaining part of the proof now follows with minor modifications from the proof of Theorem~\ref{thm:W2Whom} by noting that $\hat q$ is bounded away from $e$, and hence the preimage $f^{-1}(K)$ is bounded. Therefore the limit of a sequence in the preimage must belong to $f^{-1}(K)$, and hence $f^{-1}(K)$ is compact as claimed.
%\end{proofWithName}
\end{proof}

\bibliographystyle{siam}
\bibliography{ref}

\end{document}

%% file: figures/blockdiagram.tex
\tikzstyle{int}=[draw, minimum size=2em]
\tikzstyle{init} = [pin edge={to-,thin,black}]

\begin{tikzpicture}[node distance=3.7cm, auto, >=latex']
    \node [int] (a) {\text{Linear system}};
    \node (b) [left of=a, node distance=2.3cm, coordinate] {a};
    \node [int] (c) [right of=a] {\text{Static nonlinearity}};
    \node [coordinate] (end) [right of=c, node distance=2.5cm]{};
    \path[->] (b) edge node {$u_\tb$} (a);
    \path[->] (a) edge node {$x_\tb$} (c);
    \draw[->] (c) edge node {$y_\tb$} (end) ;
\end{tikzpicture}

%\tikzstyle{int}=[draw, fill=blue!20, minimum size=2em]
%\tikzstyle{init} = [pin edge={to-,thin,black}]
%
%\begin{tikzpicture}[node distance=2.5cm,auto,>=latex']
%    \node [int, pin={[init]above:$v_0$}] (a) {$\frac{1}{s}$};
%    \node (b) [left of=a,node distance=2cm, coordinate] {a};
%    \node [int, pin={[init]above:$p_0$}] (c) [right of=a] {$\frac{1}{s}$};
%    \node [coordinate] (end) [right of=c, node distance=2cm]{};
%    \path[->] (b) edge node {$a$} (a);
%    \path[->] (a) edge node {$v$} (c);
%    \draw[->] (c) edge node {$p$} (end) ;
%\end{tikzpicture}

%\begin{tikzpicture}%[scale=3]
%\node [above] at (0.7,0.7) {u};
%\draw [->] (0,0.5) -- (1.4,0.5);
%
%\draw [black] (1.5,0) rectangle (3,1);
%\node at (2.25,0.5) {W(z)};
%
%\node [above] at (3.8,0.65) {x};
%\draw [->] (3.1,0.5) -- (4.5,0.5);
%
%\draw [black] (4.6,0) rectangle (6.1,1);
%\node at (5.35,0.5) {\text{Thresholding}};
%
%\node [above] at (6.9,0.65) {y};
%\draw [->] (6.2,0.5) -- (7.7,0.5);
%\end{tikzpicture}